\documentclass[letterpaper, 11pt,  reqno]{amsart}
\usepackage[margin=1.2in, marginparwidth=1.5cm, marginparsep=0.5cm]{geometry}

\usepackage{amsmath,amssymb,amscd,amsthm,amsxtra, esint}

\usepackage[implicit=true]{hyperref}

\usepackage{color}

\allowdisplaybreaks[2]

\definecolor{gr}{rgb}   {0.,   0.69,   0.23 }
\definecolor{bl}{rgb}   {0.,   0.5,   1. }
\definecolor{mg}{rgb}   {0.85,  0.,    0.85}
\definecolor{yl}{rgb}   {0.8,  0.7,   0.}
\definecolor{or}{rgb}  {0.7,0.2,0.2}

\newtheorem{theorem}{Theorem} [section]
\newtheorem{maintheorem}{Theorem}
\newtheorem{lemma}[theorem]{Lemma}
\newtheorem{proposition}[theorem]{Proposition}
\newtheorem{remark}[theorem]{Remark}

\DeclareMathOperator*{\supp}{supp}

\newcommand{\I}{\hspace{0.5mm}\text{I}\hspace{0.5mm}}
\newcommand{\II}{\text{I \hspace{-2.8mm} I} }
\newcommand{\III}{\text{I \hspace{-2.9mm} I \hspace{-2.9mm} I}}

\newcommand{\IV}{\text{I \hspace{-2.9mm} V}}

\newcommand{\noi}{\noindent}
\newcommand{\Z}{\mathbb{Z}}
\newcommand{\R}{\mathbb{R}}

\newcommand{\T}{\mathbb{T}}

\let\Re=\undefined\DeclareMathOperator*{\Re}{Re}

\let\P= \undefined
\newcommand{\P}{\mathbf{P}}

\newcommand{\E}{\mathbb{E}}

\renewcommand{\L}{\mathcal{L}}

\newcommand{\F}{\mathcal{F}}

\newcommand{\al}{\alpha}
\newcommand{\be}{\beta}
\newcommand{\dl}{\delta}

\newcommand{\nb}{\nabla}

\newcommand{\Dl}{\Delta}
\newcommand{\eps}{\varepsilon}
\newcommand{\kk}{\kappa}

\newcommand{\ld}{\lambda}
\newcommand{\Ld}{\Lambda}
\newcommand{\s}{\sigma}
\newcommand{\Si}{\Sigma}
\newcommand{\ft}{\widehat}

\newcommand{\wt}{\widetilde}
\newcommand{\cj}{\overline}

\newcommand{\dt}{\partial_t}

\newcommand{\LRA}{\Longrightarrow}

\newcommand{\ta}{\theta}

\renewcommand{\l}{\ell}
\renewcommand{\o}{\omega}
\renewcommand{\O}{\Omega}

\newcommand{\les}{\lesssim}
\newcommand{\ges}{\gtrsim}

\newcommand{\jb}[1]
{\langle #1 \rangle}

\newcommand{\N}{\mathbb{N}}

\renewcommand{\H}{\mathcal{H}}

\newtheorem*{ackno}{Acknowledgements}

\numberwithin{equation}{section}
\numberwithin{theorem}{section}

\newcommand{\nbn}{\jb{n}_{\hspace{-0.5mm}_N}}

\newcommand{\wz}{{\wt z}}
\newcommand{\wZ}{{\wt Z}}
\newcommand{\ws}{{\wt \s}}

\usepackage{tikz} 
\usepackage{marginnote}
\usepackage{scalerel} 

\usetikzlibrary{shapes.misc}
\usetikzlibrary{shapes.symbols}
\usetikzlibrary{decorations}
\usetikzlibrary{decorations.markings}

\tikzset{
	dot/.style={circle,fill=black,draw=black,inner sep=0pt,minimum size=0.5mm},
	>=stealth,
	}

\tikzset{
	dot2/.style={circle,fill=black,draw=black,inner sep=0pt,minimum size=0.2mm},
	>=stealth,
	}

\tikzset{
	ddot/.style={circle,fill=white,draw=black,inner sep=0pt,minimum size=0.8mm},
	>=stealth,
	}

\tikzset{decision/.style={ 
        draw,
        diamond,
        aspect=1.5
    }}

\tikzset{dia2/.style
={diamond,fill=white,draw=black,inner sep=0pt,minimum size=1mm},
	>=stealth,
	}

\tikzset{dia/.style
={star,fill=black,draw=black,inner sep=0pt,minimum size=1mm},
	>=stealth,
	}

\tikzset{dia/.style
={diamond,fill=black,draw=black,inner sep=0pt,minimum size=1.3mm},
	>=stealth,
	}

\makeatletter

\def\<#1>{\xusebox{#1}}
\makeatother

\newcommand{\pe}{\mathbin{\scaleobj{0.7}{\tikz \draw (0,0) node[shape=circle,draw,inner sep=0pt,minimum size=8.5pt] {\scriptsize  $=$};}}}

\newcommand{\pl}{\mathbin{\scaleobj{0.7}{\tikz \draw (0,0) node[shape=circle,draw,inner sep=0pt,minimum size=8.5pt] {\scriptsize $<$};}}}
\newcommand{\pg}{\mathbin{\scaleobj{0.7}{\tikz \draw (0,0) node[shape=circle,draw,inner sep=0pt,minimum size=8.5pt] {\scriptsize $>$};}}}

\newcommand{\too}{\longrightarrow}

\begin{document}
\baselineskip = 14pt

\title[The cubic NLW in negative Sobolev spaces]
{
Probabilistic local well-posedness of the cubic nonlinear wave equation in negative Sobolev spaces
} 
\author[T.~Oh , O.~Pocovnicu, and   N.~Tzvetkov]
{Tadahiro Oh, Oana Pocovnicu,  and Nikolay Tzvetkov}

\address{
Tadahiro Oh, School of Mathematics\\
The University of Edinburgh\\
and The Maxwell Institute for the Mathematical Sciences\\
James Clerk Maxwell Building\\
The King's Buildings\\
Peter Guthrie Tait Road\\
Edinburgh\\ 
EH9 3FD\\
 United Kingdom}

\email{hiro.oh@ed.ac.uk}

\address{
Oana Pocovnicu\\
Department of Mathematics, Heriot-Watt University and The Maxwell Institute for the Mathematical Sciences, Edinburgh, EH14 4AS, United Kingdom}
\email{o.pocovnicu@hw.ac.uk}

\address{
Nikolay Tzvetkov\\
CY Cergy Paris University\\
Cergy-Pontoise\\ 
F-95000\\
UMR 8088 du CNRS\\ France}

\email{nikolay.tzvetkov@cyu.fr}


%

\dedicatory{Dedicated to the memory of Professor Ioan I.\,Vrabie (1951--2017)}

\begin{abstract}
We study  the three-dimensional 
 cubic nonlinear wave equation (NLW) with random initial data below $L^2(\T^3)$.
By considering the second order expansion
in terms of the random linear solution, 
we prove almost sure local well-posedness of the renormalized NLW
in negative Sobolev spaces.
We also prove
a  new instability  result  for the defocusing cubic NLW without renormalization 
in negative Sobolev spaces, 
which is in the spirit of the so-called triviality in the study of stochastic partial differential equations.
More precisely, by studying (un-renormalized) NLW with 
given  smooth deterministic initial data plus a certain truncated random initial data, 
we show that, 
as the truncation is removed, 
the solutions converge to $0$ in the distributional sense
for {\it any} deterministic initial data.

\end{abstract}

%
\maketitle

\tableofcontents

\section{Introduction}
\subsection{Main result}
We consider the Cauchy problem for the defocusing cubic nonlinear wave equation (NLW)
on the three-dimensional torus $\T^3 = (\R/2\pi \Z)^3$:
\begin{equation}\label{KG}
\begin{cases}
\partial_t^2 u-\Delta u+u^3=0\\
(u, \dt u) |_{t = 0}  = (u_0,u_1)\in \H^s(\T^3), 
\end{cases}
\end{equation}

\noi
where $u :\R\times \T^3\to  \R$
and 
$\H^s(\T^3) = H^s(\T^3)\times H^{s-1}(\T^3)$.
Here, 
$H^s(\T^3)$ denotes the standard Sobolev space on $\T^3$
 endowed with the norm:
\[
\|f\|_{H^s(\T^3)}=\|\jb{n}^s\ft{f}(n)\|_{\l^2(\Z^3)},
\]

\noi
where $\ft{u}(n)$ is the Fourier coefficient of $u$ and  $\jb{\,\cdot\,} = (1+|\,\cdot\,|^2)^\frac{1}{2}$ .
The classical well-posedness result (see for example \cite {tz-cime})  for \eqref{KG} reads as follows.

\begin{maintheorem}\label{THM:0}
Let $s\geq 1$. Then, for every $(u_0,u_1)\in \H^s(\T^3)$,
 there exists a unique global-in-time solution $u$ to~\eqref{KG} in $C(\R;H^s(\T^3))$. Moreover, the dependence 
of the solution map$: (u_0, u_1) \mapsto u(t)$ on initial data and  time $t \in \R$ is continuous. 
\end{maintheorem}

The proof of Theorem \ref{THM:0} follows from Sobolev's inequality: $H^1(\T^3) \subset L^6(\T^3)$
and the conservation of the energy for \eqref{KG}.
Recall that the scaling symmetry: $u (t, x) \mapsto \ld u(\ld t, \ld x)$ for \eqref{KG}  posed on $\R^3$
induces the scaling-critical Sobolev regularity $s_\text{crit} =\frac 12$.
By using the Strichartz estimates (see Lemma \ref{LEM:Str} below), 
one may indeed show that  the Cauchy problem~\eqref{KG} remains locally well-posed 
in $\H^s(\T^3)$
for $s\geq \frac 12$ \cite{LS}.
On the other hand, it is known that  the Cauchy problem~\eqref{KG} is ill-posed
for  $s <\frac 12$  \cite{CCT, BT1, Xia, OOTz, FOk}.
 We refer to~\cite {tz-cime, OOTz, FOk}~for the proofs of these facts.

 One may then ask whether a sort of well-posedness of  \eqref{KG}  survives 
 below the scaling-critical regularity, i.e.~for $s<\frac 12$. 
 As it was shown in the work \cite{BT1, BT-JEMS}
 by Burq and the third author, 
  the answer to this question is positive 
 if one considers {\it random} initial data.
 In this paper, we will primarily consider the following random initial data: 
\begin{equation}\label{series}
u_0^\o = \sum_{n \in \Z^3} \frac{g_n(\o)}{\jb{n}^{\al}}e^{in\cdot x}
\qquad\text{and}\qquad 
u_1^\o = \sum_{n \in \Z^3} \frac{h_n(\o)}{\jb{n}^{\al-1}}e^{in\cdot x}, 
\end{equation}

\noi
where the series   $\{ g_n \}_{n \in \Z^3}$ and  $\{ h_n \}_{n \in \Z^3}$ are two families of 
independent standard   complex-valued  Gaussian random variables  
on a probability space $(\O, \F, P)$ conditioned that\footnote{In particular, 
$g_0$ and $h_0$ are real-valued.}  $g_n=\overline{g_{-n}}$,  $h_n=\overline{h_{-n}}$, 
$n \in \Z^3$. More precisely,  with the notation $\N=\{1,2,\,3, \cdots\}$, we first define the  index set $\Ld$ by 
\begin{align}
 \Ld = 
 (\Z^2\times \N) \cup
 (\Z\times \N\times\{0\}) \cup (\N\times \{(0,0)\}) \cup\{(0, 0,0)\}.
\label{index}
\end{align}
 
 \noi
We then define  $\{g_n, h_n\}_{n\in\Ld}$  to be a family   of independent standard   Gaussian random variables 
which are complex-valued for $n\neq 0$ and are real-valued for $n=0$.  
We finally set  $g_n=\overline{g_{-n}}$,  $h_n=\overline{h_{-n}}$ for $n \in \Z^3 \setminus  \Ld$.

The partial sums for the series $(u_0^\o, u_1^\o)$ in \eqref{series} 
form a  Cauchy sequence in $L^2(\O; \H^{s}(\T^3) )$ 
for every $s<\al-\frac{3}{2}$
and  therefore the random initial data 
$(u_0^\o, u_1^\o)$ in \eqref{series} 
 belongs almost surely to $\H^s(\T^3)$
 for the same range of $s$.
On the other hand, one may show that  the probability of the event 
$(u_0^\o,u_1^\o)\in \H^{\al-\frac{3}{2}}(\T^3)$ is zero.
See Lemma B.1 in \cite{BT1}.
As a result,  when $\al>\frac 52$, 
  one may apply the classical global well-posedness result in Theorem~\ref{THM:0} 
  for the random initial data $(u_0^\o,u_1^\o)$ given by \eqref{series}
  since $(u_0^\o, u_1^\o) \in \H^1(\T^3)$ almost surely.
For $\al > 2$, one may still apply the  more refined (deterministic) local well-posedness result in $\H^\frac{1}{2}(\T^3)$
mentioned above.
 For $\al \leq 2$, 
however,  the Cauchy problem 
\eqref{KG}  becomes ill-posed.
 Despite this ill-posedness result, the analysis in~\cite{BT-JEMS,tz-cime} implies the following statement.

 \begin{maintheorem}\label{THM:BT}
Let $\al>\frac 32$ and $s<\al- \frac 32$. 
Let $\{u_N\}_{N \in \N}$ 
be a sequence of the smooth global solutions\footnote{Theorem ~\ref{THM:0}
guarantees existence of smooth global solutions $\{u_N\}_{N\in \N}$  to \eqref{KG}.}
to~\eqref{KG} with 
the following random $C^\infty$-initial data: 
\begin{equation}\label{series_N}
u_{0, N}^\o(x) = \sum_{|n|\leq N } \frac{g_n(\o)}{\jb{n}^{\al}}e^{in\cdot x}
\qquad\text{and}\qquad 
u_{1, N}^\o(x) = \sum_{|n|\leq N} \frac{h_n(\o)}{\jb{n}^{\al-1}}e^{in\cdot x},
\end{equation}

\noi
where
$\{ g_n \}_{n \in \Z^3}$ and  $\{ h_n \}_{n \in \Z^3}$ are as in \eqref{series}.
Then, 
as $N \to \infty$,  
 $u_N$  converges  almost surely  to a \textup{(}unique\textup{)} limit $u$ 
  in $C(\R;H^s(\T^3))$, 
satisfying NLW \eqref{KG} in a distributional sense. 
\end{maintheorem}

Here, by uniqueness, we firstly mean that the entire sequence $\{u_N \}_{N\in \N}$ converges
to $u$, not up to some subsequence.
Compare this with 
  the case of weak solution techniques
  (see for example \cite{BTTz, BTT1}), which usually  only give convergence
  up to subsequences.
  Furthermore, 
when we regularize the random initial data $(u_0^\o, u_1^\o)$
 in \eqref{series} by mollification,  
it can be shown that   the limit $u$ is independent
of the choice of mollification kernels.  See Remark \ref{REM:uniq}.
Lastly, as we see in  Subsection~\ref{SUBSEC:fac}, 
 the limit $u$ admits a decomposition
 $u = z_1 + v$, where 
 $z_1$ is the random linear solution, emanating from 
  the random initial data $(u_0^\o, u_1^\o)$, 
and  $v$ is the {\it unique} solution to the perturbed NLW:
\begin{align*}
\begin{cases}
\L v + (v+ z_1)^3 = 0\\
(v, \dt v)|_{ t= 0} = (0, 0),
\end{cases}
\end{align*}

\noi
Similar comments apply to the limiting distribution $u$ in Theorem \ref{THM:3}
below.

For $\al \leq \frac 32$, 
$u_0^\o$ in \eqref{series} is  almost surely no longer a classical function and it should be interpreted 
as a random Schwartz distribution lying in a Sobolev space of negative index. 
Therefore for $\al\leq \frac 32$,
 the study of \eqref{KG} with the random initial data \eqref{series} is no longer 
 within the scope of applicability of  \cite{BT-JEMS,tz-cime}.
The goal of this paper is to 
extend  the results in  \cite{BT-JEMS,tz-cime} to the random initial data 
when they are no longer classical functions.
More precisely, 
we prove the following statement.

\begin{maintheorem}\label{THM:3}
Let $\frac 54 < \al \leq \frac{3}{2}$ and $s<\al-\frac 32$. 
There exists 
a divergent sequence 
$\{\al_N\}_{N \in \N}$ of  positive numbers such that the following holds true;
there exist small  $T_0>0$ 
and positive constants $C$, $c$, $\kappa$
such that for every $T\in (0,T_0]$, 
 there exists a set $\O_T$ of complemental probability smaller than $C\exp(-c/T^\kk)$ 
 such that  if we denote by   $\{u_N\}_{N\in \N}$ 
the  smooth global   solutions to 
\begin{equation}\label{KG_N}
\begin{cases}
\partial_t^2 u_N-\Delta u_N+u_N^3 - \al_N u_N=0\\
(u_N, \dt u_N)|_{t = 0} = (u^\o_{0, N}, u^\o_{1, N}), 
\end{cases}
\end{equation}

\noi
where the random initial data  $(u_{0, N}^\o, u_{1, N}^\o)$ 
is given 
by the truncated Fourier series in  \eqref{series_N},
 then for every $\o\in \O_T$,
 the sequence $\{u_N\}_{N\in \N}$  converges
 to some \textup{(}unique\textup{)} limiting distribution $u$
 in $C([-T,T];H^s(\T^3))$ as $N \to \infty$.
\end{maintheorem}

We prove Theorem \ref{THM:3}
by writing $u_N$ in the second order expansion:
$u_N = z_{1, N} + z_{2, N} + w_N$, 
where  $z_{1, N}$ is the  linear solution,\footnote{For a technical reason, 
we take $z_{1, N}$ to satisfy the linear Klein-Gordon equation.
See \eqref{Z1} below.} emanating from 
  the random initial data $(u_0^\o, u_1^\o)$, 
and $z_{2, N}$ denotes the additional term
appearing in the Picard second iterate; see \eqref{Z4}
below.
We first use stochastic analysis to show convergence of $z_{j, N}$, $j = 1, 2$, 
and then show convergence of the residual term $w_N$ by deterministic analysis.

In view of the asymptotic behavior $\al_N \to \infty$, 
one may be tempted to say that the limiting distribution $u$ 
obtained in Theorem~\ref{THM:3}
is a solution 
to the following limit ``equation'':
\begin{equation*} 
\begin{cases}
\partial_t^2 u-\Delta u+u^3 - \infty\cdot  u=0\\
(u, \dt u)|_{t = 0} = (u^\o_{0}, u^\o_{1}), 
\end{cases}
\end{equation*}

\noi
where the random initial data $(u^\o_{0}, u^\o_{1})$ is
as in \eqref{series}.
The expression $\infty \cdot u$
is merely formal and thus 
a natural question is to understand in which sense
$u$ satisfies the cubic NLW on~$\T^3$.
As we see in the next subsection, 
the limit $u$ has the decomposition $
u = z_1 + z_2 + w$ (see~\eqref{L5a}), 
where $z_j$, $j = 1, 2$, denotes the limit of $z_{j, N}$
in a suitable sense
and the residual term $w$ satisfies
the perturbed NLW equation; see  \eqref{KG10} below.
The  uniqueness statement in Theorem~\ref{THM:3} refers
to the uniqueness of $w$ as a solution to this perturbed NLW equation
(along with the uniqueness of various stochastic terms appearing in \eqref{KG10} 
as the limits of their regularized versions); see Remark~\ref{REM:uniq2}.
See also Remark~\ref{REM:uniq} below.
We will describe our strategy  in the next two subsections.
We also refer readers to \cite{OTh2} for a related discussion in the two-dimensional case.

Given fixed $N\in \N$, 
by adapting the classical argument, it is easy to see that 
the truncated equation \eqref{KG_N} is globally well-posed
 in $\H^s(\T^3)$ for $s\geq 1$.
In particular, one needs to apply a Gronwall-type argument to exclude 
a possible finite-time blowup of 
the $\H^1$-norm of a  solution.
The main issue here is that  
there is no good {\it uniform} (in $N$) bound
for the solutions to \eqref{KG_N}.
One may try to extend the local-in-time solutions constructed  in Theorem~\ref{THM:3} globally in time by using truncated energies in the spirit of the  $I$-method, introduced in \cite{CKSTT}. 
See \cite{GKOT}
for such a globalization argument in the context of the two-dimensional stochastic NLW.

Our ultimate goal is to push the analysis in the proof of Theorem~\ref{THM:3} 
to  cover the case $\al=1$, corresponding to the regularity of
the natural Gibbs measure associated with the cubic NLW. 
In the field of singular stochastic parabolic PDEs, there has been a
significant progress in recent years.
In particular, a substantial effort 
\cite{H, GIP, CC, HairerM, MW, AK}
was made to give a proper meaning to the 
 stochastic quantization equation (SQE) on $\T^3$, formally written as
\begin{align}
\dt u -  \Dl u  = - u^3 + \infty \cdot u  + \xi.
\label{SQE1}
\end{align}

\noi
Here,    $\xi$ denotes the so-called space-time white noise.
On the one hand, 
the randomization effects in the present paper are close in spirit to the works cited above.
On the other hand, 
  the deterministic part of the analysis in the context of the heat and the wave equations represent significant differences because, as it is well known, the deterministic regularity theories for these two types of equations are quite different.  
In fact, in order to extend Theorem~\ref{THM:3} to lower values of $\al$, 
it is crucial to exploit dispersion at a multilinear level, 
a consideration specific to dispersive equations, 
and combine it with randomization effects.
See, for example,  a recent work \cite{GKO2} by Gubinelli, Koch, and the first author
on the three-dimensional stochastic NLW with a quadratic nonlinearity.
Furthermore, 
in order to treat lower values of $\al$, 
it will be crucial  to impose a  structure on the residual part~$w$.
See Remark~\ref{REM:WP} for a further discussion.

\begin{remark}\label{REM:uniq}\rm

We say that  $\eta \in C(\R^3 ; [0, 1])$ is a mollification kernel
if $\int \eta dx = 1$ and $\supp \eta \subset (-\pi, \pi]^3\simeq \T^3$.
Given a mollification kernel $\eta$, define $\eta_\eps$ by setting
$\eta_\eps(x) = \eps^{-3} \eta(\eps^{-1} x)$.
Then, $\{\eta_\eps\}_{0 < \eps \leq 1}$ forms an approximate identity on $\T^3$.
By slightly modifying the proof of Theorem \ref{THM:BT}, 
we can show that 
if we denote by $u_\eps$, the solution to \eqref{KG}
with the initial data $(\eta_\eps * u_0^\o, \eta_\eps * u_1^\o)$, 
where $(u_0^\o, u_1^\o)$ is as in \eqref{series}, 
then, for $\al > \frac 32$ and $s < \al - \frac 32$, 
 $u_\eps$  converges in probability   to some (unique) limit $u$ 
  in $C(\R;H^s(\T^3))$ as $\eps \to 0$.
 Here, the limit $u$ is independent of the choice of mollification kernels $\eta$.
Similarly, when $\frac 54 < \al \leq \frac 32$, 
a  slight modification of the proof of Theorem \ref{THM:3}
shows that
there exists
a divergent sequence $\al_\eps$ (as $\eps \to 0$) such that 
the solution $u_\eps$ to 
\begin{equation*}
\begin{cases}
\partial_t^2 u_\eps-\Delta u_\eps+u_\eps^3 - \al_\eps u_\eps=0\\
(u_\eps, \dt u_\eps)|_{t = 0} = (\eta_\eps* u^\o_{0}, \eta_\eps * u^\o_{1}) 
\end{cases}
\end{equation*}

\noi
converges in probability to some (unique) limit $u$
  in $C([-T_\o, T_\o];H^s(\T^3))$, 
where  $T_\o > 0$ almost surely.
Once again, the limit $u_\eps$ is independent of the choice of mollification kernels~$\eta$.

\end{remark}

\begin{remark}\rm
As in \cite{BT1, BT-JEMS}, it is possible to consider a more general class of random initial data.
Let a deterministic pair $(u_0, u_1) \in \H^s(\T^3)$ 
be given by the following Fourier series:
\begin{equation*}
u_0 = \sum_{n \in \Z^3} a_n e^{in\cdot x}
\qquad\text{and}\qquad 
u_1 = \sum_{n \in \Z^3} b_n e^{in\cdot x}
\end{equation*}

\noi
with the constraint $a_{-n} = \cj{a_n}$ and $b_{-n} = \cj{b_n}$, $n \in \Z^3$.
We consider the randomized initial data $(u_0^\o, u_1^\o)$ given by 
\begin{equation*}
u_0^\o = \sum_{n \in \Z^3} g_n(\o)a_n e^{in\cdot x}
\qquad\text{and}\qquad 
u_1^\o = \sum_{n \in \Z^3}h_n(\o)b_n e^{in\cdot x}, 
\end{equation*}

\noi
Then, by slightly modifying the proof of Theorem \ref{THM:3}, 
it is easy to see that,
for $s > - \frac 16$ (corresponding to $\al > \frac 43$ in \eqref{series}),  we can introduce a  time dependent  divergent sequence 
$\{\al_N\}_{N \in \N}$ with $\al_N = \al_N(t)$ such that the solution $u_N$
to \eqref{KG_N} converges to some (unique) limit $u$ in 
$C([-T_\o, T_\o]; H^s(\T^3))$, 
where $T_\o > 0$ almost surely.
For this range of $s$, we need only the first order expansion.
See the next subsection.
For lower values of $s$, 
one may need to impose some additional summability assumptions
on $\{a_n\}_{n \in \Z^3}$ and $\{b_n\}_{n \in \Z^3}$
(in particular to replicate the proof of Proposition \ref{PROP:Z4}
to obtain an analogue of Theorem \ref{THM:3}).

\end{remark}

\subsection{Outline of the proof of Theorem \ref{THM:3}.}
\label{SUBSEC:outline}
In the following, we present the main idea of the proof of Theorem \ref{THM:3}.
Fix  $\al\leq\frac{3}{2}$. 
With the short-hand notation:\footnote {For our subsequent analysis, it will be more convenient to study the linear Klein-Gordon equation 
rather than the linear wave equation. 
}
\begin{align}
\L:=\partial_t^2 -\Delta +1, 
\label{lin0}
\end{align}

\noi
we denote 
by $z_{1,N}=z_{1,N}(t,x,\o)$ the solution  to the following linear Klein-Gordon equation:
\begin{align}
\L z_{1,N}(t,x,\o)=0
\label{Z1}
\end{align}

\noi
with the random initial data 
$(u_{0, N}^\o, u_{1, N}^\o)$ given by the truncated Fourier series in~\eqref{series_N}.
In the following, we discuss spatial regularities of various stochastic terms
for fixed $t \in \R$.
For simplicity of notation, we suppress the $t$-dependence
and discuss spatial regularities.
It is easy to see from \eqref{series_N} that
$z_{1, N}$ converges almost surely 
to some limit $z_{1}$ 
in $H^{s_1}(\T^3)$
as $N \to \infty$,
provided that 
\begin{align}
s_1 < \al - \frac 32.
\label{S1}
\end{align}

\noi
In particular, when $\al \leq \frac 32$, 
$z_{1, N}$  has
negative  Sobolev regularity (in the limiting sense)
and thus 
  $(z_{1,N})^2$ and  $(z_{1,N})^3$
 do not have well-defined limits (in any topology) as $N \to \infty$
 since it involves products of two distributions
 of negative regularities.

Let $u_N$ be the solution to the renormalized NLW \eqref{KG_N}
with the same truncated random initial data 
$(u_{0, N}^\o, u_{1, N}^\o)$  in~\eqref{series_N}.
By writing  $u$ as  
\begin{align}
u_N=z_{1,N}+v_N,
\label{Z1a} 
\end{align}

\noi
we see that the residual term $v_N = u_N - z_{1, N}$ satisfies the following equation:
\begin{equation}\label{L1}
\begin{cases}
\L v_N+ v_N^3 + 
3z_{1,N} v_N^2+
3\big\{(z_{1,N})^2-\s_{N}\big\}v_N
+ \big\{(z_{1,N})^3-3 \s_N z_{1,N}\big\}=0\\
(v_N, \dt v_N ) |_{t=0}=(0, 0), 
\end{cases}
\end{equation}

\noi
where 
the parameter $\s_N$ is defined by 
\begin{align*}
\s_N:=\frac{\al_N+1}{3}.
\end{align*}

\noi
As it is well known, 
 the key point in the equation \eqref{L1} is that 
the terms 
\begin{align}
 Z_{2, N} := (z_{1,N})^2-\s_{N}
\qquad  \text{and} \qquad 
Z_{3, N} := (z_{1,N})^3-3\s_N z_{1,N}
\label{L1a}
\end{align}

\noi
  are ``renormalizations" of 
$(z_{1,N})^2$ and 
$(z_{1,N})^3$. 
Here, by ``renormalizations",
 we mean that  by choosing a suitable renormalization constant $\s_N$, 
the terms $Z_{2, N}$ and $Z_{3, N}$
converge almost surely in suitable negative Sobolev spaces 
as $N \to \infty$.

The regularity $s_1 < \al - \frac 32$ of $z_{1,N}$ (in the limit)
and a simple paraproduct computation
show that if the expressions 
$Z_{2, N} = (z_{1,N})^2-\s_{N}$
and 
$Z_{3, N} = (z_{1,N})^3-3\s_N z_{1,N}$
have any well-defined limits as $N \to \infty$, 
then their regularities in the limit
are expected to be 
\begin{align}
 s_2 <2\bigg(\al-\frac 32\bigg)
\qquad \text{and}
\qquad 
s_3 <3\bigg(\al-\frac 32\bigg),
\label{S2}
\end{align}

\noi
 respectively.
In fact, 
by choosing the renormalization constant $\s_N$ as
\begin{align}
\s_N : = \E\Big[\big(z_{1,N}(t,x,\o)\big)^2 \Big], 
\label{sig}
\end{align}

\noi
we show that $Z_{j, N}$ converges  in $H^{s_j}(\T^3)$
almost surely.
See Proposition \ref{PROP:ran1}.
Note that 
the renormalization constant $\s_N$ a priori depends on $t,x$ 
but it turns out to be independent of $t$ and $x$.\footnote {While we show this fact by a direct computation in \eqref{CN}, 
it can be   seen from the stationarity (in both $t$ and $x$) of the stochastic process
$\{ z_{1, N} (t, x) \}_{(t, x) \in \R \times \T^3}$.  See Remark \ref{REM:sig}.}
We will also see that,  for $N\gg 1$,  $\s_N$ behaves like 
(i) $\sim N^{3-2\al}$ 
when $\al < \frac 32$ and (ii) $\sim \log N$ when $\al = \frac 32$.
See \eqref{CN} below.

\medskip

Thanks to the Strichartz estimates (see Lemma \ref{LEM:Str} below),  the deterministic Cauchy problem 
for 
\[\L v+v^3=0\] 

\noi
is locally well-posed 
in $\H^s(\T^3)$ for  $s\geq \frac 12$.
We may therefore hope to solve the equation~\eqref{L1} 
uniformly in $N \in \N$ by 
the method of \cite{BO96,BT1,BT-JEMS},
 if we can ensure that the solution $v_N$ to  the following linear problem: 
\begin{equation}\label{L2}
\L v_N+\big\{(z_{1,N})^3-3\s_N z_{1,N}\big\}=0
\end{equation}

\noi
with the zero initial data $(v_N, \dt v_N ) |_{t=0}=(0, 0)$
remains bounded in $H^{\frac 12}(\T^3)$ 
as $N \to \infty$. 
Using one degree of smoothing under the wave Duhamel operator
 (see \eqref{basic_reg} below), 
we see that the solution to \eqref{L2} is almost surely  bounded  in $H^\frac{1}{2}(\T^3)$
uniformly in $N \in \N$, provided 
\[3\bigg(\al-\frac{3}{2}\bigg)+1> \frac{1}{2}
\quad \LRA \quad\al>\frac{4}{3}.\]

\noi
Therefore, 
 $\al=\frac{4}{3}$ seems to be the limit of the approach of \cite{BO96,BT1,BT-JEMS}.\footnote{Here, 
 we are not taking into account  a possible multilinear smoothing
 for the solution $v$ to \eqref{L2}.}

 \medskip
 
 In order to go below the $\al=\frac{4}{3}$ threshold,
  a new argument is needed. 
The introduction of such an argument is the main idea of this paper. 
More precisely, we further decompose $v_N$ in \eqref{Z1a}
as 
\begin{align}
v_N=z_{2,N}+w_N
\label{Z4a}
\end{align} 

\noi
for some residual term $w_N$, 
where $z_{2,N}$ is the solution to the following equation:
\begin{align}
\begin{cases}
\L z_{2,N}+\big\{(z_{1,N})^3-3\s_N z_{1,N}\big\}=0\\
(z_{2,N}, \dt z_{2, N} ) |_{t=0}=(0, 0).
\end{cases}
\label{Z4}
\end{align}

\noi
Thanks to the one degree of smoothing, 
we see that $z_{2, N}$ converges to some limit in $H^s(\T^3)$, 
provided that 
\begin{align*}
 s = s_3 + 1 < 3\bigg(\al-\frac{3}{2}\bigg)+1
\end{align*}

\noi
 In terms of the original solution $u_N$ to \eqref{KG_N},
 we have from \eqref{Z1a} and \eqref{Z4a}
 that 
\begin{align} 
u_N=z_{1,N}+z_{2,N}+w_N.
\label{Z5}
\end{align}

\noi
Note that 
$z_{1,N}+z_{2,N}$ corresponds to the Picard second iterate for the truncated
renormalized equation \eqref{KG_N}.

The equation for $w_N$ can now be written as 
 \begin{equation}\label{L3}
\hspace{-2mm}
\begin{cases}
\L w_N+ (w_N + z_{2,N})^3
+ 3z_{1,N} (w_N + z_{2,N})^2
+ 3\big\{(z_{1,N})^2-\s_{N}\big\}(w_N+z_{2,N})=0\\
(w_N, \dt w_N ) |_{t=0}=(0, 0).
\end{cases}
\end{equation}

\noi
By using the second order expansion \eqref{Z5}, 
we have eliminated the most singular term  $Z_{3, N} = (z_{1,N})^3-3\s_N z_{1,N}$ in \eqref{L1}.
In the equation \eqref{L3},  there are several source terms\footnote {Namely, 
purely stochastic terms independent of the unknown $w_N$.} 
and they are precisely 
the quintic, septic, and nonic (i.e.~degree nine) terms added
in considering the Picard third iterate for \eqref{KG_N}.
As we see below,  the most singular term in~\eqref{L3} is 
the following quintic term:
 \begin{equation}\label{L4}
Z_{5, N} :=  3\big\{(z_{1,N})^2-\s_{N}\big\} z_{2,N}, 
\end{equation}

\noi
 where $z_{2, N}$ is the solution to \eqref{Z4}.
As we already mentioned, 
the term $Z_{2, N} = (z_{1,N})^2-\s_{N}$
and 
the second order term  $z_{2,N}$ pass to the limits in $H^s(\T^3)$ for 
$s < 2(\al - \frac 32)$ and 
$s<3(\al-\frac{3}{2})+1$, respectively.
In order to make sense of the product of $Z_{2, N}$ and $z_{2, N}$
in~\eqref{L4}
by deterministic paradifferential calculus 
(see Lemma \ref{LEM:para} below), we need the sum
of the two regularities to be positive, namely
\[  2\bigg(\al - \frac 32\bigg) + 3\bigg(\al-\frac{3}{2}\bigg)+1 >0
\quad \LRA \quad  \al > \frac{13}{10}.\]

\noi
Otherwise, i.e.~for $\al \leq \frac{13}{10}$, 
we will need to make sense of the product \eqref{L4},
using stochastic analysis.
See Proposition \ref{PROP:Z4}.
In either case, 
when the second factor in \eqref{L4} has positive regularity 
$3(\al-\frac{3}{2})+1>0$, i.e.~$\al > \frac76$, 
we show that the product \eqref{L4}
(in the limit) inherits the regularity
from $Z_{2, N} = (z_{1,N})^2-\s_{N}$, 
allowing us to pass to a limit in $H^s(\T^3)$ for
\begin{align*}
s < 2\bigg(\al-\frac{3}{2}\bigg).
\end{align*}

Once we are able to pass the term $Z_{5, N}$ in \eqref{L4}  in the limit $N\rightarrow\infty$,
the main issue in solving the equation \eqref{L3} for $w_N$ by the deterministic Strichartz theory 
is to ensure that the solution of
\begin{equation}\label{L5}
\begin{cases}
\L w+3\big\{(z_{1,N})^2-\s_{N}\big\} z_{2,N}=0\\
(w, \dt w) |_{t=0}=(0, 0)
\end{cases}
\end{equation}

\noi
remains bounded in $H^\frac{1}{2}(\T^3)$ as $N \to \infty$  (recall that $s=\frac 12$ is the threshold regularity for  the deterministic local well-posedness theory for the  cubic wave equation on $\T^3$). 
Using  again one degree of smoothing under the wave Duhamel operator, 
we see that the solution to \eqref{L5} is almost surely bounded 
in $H^\frac{1}{2}(\T^3)$,  provided 
\[
2\bigg(\al-\frac{3}{2}\bigg)+1>\frac{1}{2}
\quad \LRA\quad  \al>\frac{5}{4}.
\]

\noi
This  explains the restriction $\al > \frac 54$ in  Theorem~\ref{THM:3}.  
See also Remark \ref{REM:WP}.
We  point out that  under the restriction $\al > \frac 54$, 
we can use deterministic paradifferential calculus
to  make sense of the product of $z_{1,N}$ and $z_{2,N}^2$ appearing in \eqref{L3},
uniformly in $N \in \N$.

In proving Theorem \ref{THM:3}, 
we apply the deterministic Strichartz theory and  show that $w_N$ converges almost surely to some limit $w$.
Along with the almost sure convergence of $z_{1, N}$ and $z_{2, N}$
to some limits $z_{1}$ and $z_2$, respectively, 
we conclude from the decomposition~\eqref{Z5}
that $u_N$ converges almost surely to 
\begin{align}
 u := z_1 + z_2 + w.
 \label{L5a}
\end{align}

\noi
By taking a limit of \eqref{L3} as $N \to \infty$, 
we see that $w$ is almost surely the solution to
 \begin{equation}
\begin{cases}
\L w+ (w + z_{2})^3
+ 3z_{1} (w + z_{2})^2
+ 3Z_{2}  w + 3Z_{5}=0\\
(w, \dt w ) |_{t=0}=(0, 0), 
\end{cases}
\label{KG10}
\end{equation}

\noi
where $Z_2$ and $Z_5$ are the limits of $Z_{2, N}$ in \eqref{L1a} and $Z_{5, N}$ in \eqref{L4},
respectively.
This essentially explains the proof of Theorem \ref{THM:3}.

\begin{remark}\rm

The expansion~\eqref{L5a} provides finer descriptions of $u$ at different scales;
the roughest term $z_1$ is essentially responsible for the small scale behavior of $u$, 
while $z_2$ describes its mesoscopic behavior and the smoother remainder part
$w$ describes its large-scale behavior.

\end{remark}

\begin{remark}\label{REM:WP}
\rm

The argument based on the first order expansion \eqref{Z1a}
goes back to the work of 
McKean~\cite{McKean} and 
Bourgain \cite{BO96} in the study of invariant Gibbs measures
for the nonlinear Schr\"odinger equations on $\T^d$, $d = 1, 2$.
See also \cite{BT1}.
In the field of stochastic parabolic  PDEs, 
this argument is usually referred to as the Da~Prato--Debussche trick \cite{DPD2}.

As we explained above, 
the novelty in this paper with respect to the previous work \cite{BO96,BT1,BT-JEMS} is that the proof of Theorem~\ref{THM:3} crucially relies on the second order expansion~\eqref{Z5}.
We also mention two other recent works
\cite{BOP, OTW},  where such higher order expansions were used in the context of dispersive PDEs
with random initial data. 
The higher order expansions used in~\cite{OTW}  are at negative Sobolev regularity but 
they are related to a gauge transform, which is 
very different from the situation in the present paper. 
The difference between the present paper and  \cite{BOP} is 
that, in this paper, we work in Sobolev spaces of negative indices, while 
solutions  in \cite{BOP} have positive Sobolev regularities.\footnote{In particular, 
 all the products make sense as functions in \cite{BOP}. In negative Sobolev spaces, the main problem
 is to make sense of a product as a distribution.} 
We point out that while the higher order expansions
helped lowering the regularity of random initial data in \cite{BOP}, 
the third order expansion would {\it not} help us improve Theorem \ref{THM:3}
for our problem.\footnote{That is, unless we combine it with multilinear smoothing
and imposing a further structure (such as a paracontrolled structure) on $w$.}
This can be seen from the product $Z_2 w$ in \eqref{KG10}.
From the regularity $s_2$ of $Z_2$ in \eqref{S2}
 and the regularity $\frac 12$ of~$w$, 
we see that the sum of their regularity is positive
(which is needed to make sense of the product $Z_2 w$)
only for $\al > \frac 54$.
This  is exactly the range covered in Theorem \ref{THM:3}.
Note that  a higher order expansion is used
to eliminate certain explicit stochastic terms.
Namely, even if we go into a higher order expansion, 
we can not eliminate this problematic term $Z_2 w$
since this term depends on the unknown $w$. 
In order to lower values of $\al$, 
we need to impose a structure of the residual term $w$.

For conciseness of the presentation, 
we decided to present only the simplest argument based
on the second order expansion.
There are, however, 
several ways for a possible improvement on the regularity restriction in Theorem \ref{THM:3}.
(i)~In studying the regularity and convergence properties
of  the second order stochastic term $z_{2, N}$ in \eqref{Z4}, 
we simply use a ``parabolic thinking", 
namely, we only count 
 the regularity $s_1 < \al - \frac 32$ of each of three factors $z_{1, N}$ for $Z_{3, N}$
 (modulo the renormalization)
and put them together with one degree of smoothing
coming from the wave Duhamel integral operator
{\it without} taking into account the explicit product structure
and the oscillatory nature of the linear wave propagator.
See Proposition \ref{PROP:z2} below.
In the field of dispersive PDEs, however, 
 it is crucial to 
exploit an explicit product structure
and study interaction of waves at a multilinear level
to show a further smoothing property.
See, for example,  \cite{GKO2, OOcomp, Bring2}.
In this sense, the argument presented in this paper
leaves a room for an obvious improvement.
(ii)~In recent study of singular stochastic parabolic PDEs
such as SQE \eqref{SQE1} on $\T^3$, 
higher order expansions (in terms of the stochastic forcing in the mild formulation)
were combined with the theory of regularity structures~\cite{H} or the paracontrolled calculus~\cite{CC, MW, AK}.
In fact, it is possible to employ the ideas from the paracontrolled calculus
in studying  nonlinear wave equations.
See  a recent work~\cite{GKO2} on the three-dimensional stochastic NLW
with a quadratic nonlinearity.

In a very recent preprint \cite{Bring2}
(which appeared more than one year after the appearance of the current paper), 
Bringmann studied the defocusing cubic NLW with a Hartree-type nonlinearity
on $\T^3$: 
\begin{equation}\label{KG99}
\partial_t^2 u-\Delta u+(V*u^2)u =0, 
\end{equation}

\noi
where $V = \jb{\nb}^{-\be}$ is the Bessel potential 
of order $\be > 0$.
By 
adapting the paracontrolled approach 
of \cite{GKO2} to the Hartree cubic nonlinearity
and
exploiting multilinear smoothing,\footnote{Also, combining
other tools such as the random matrix estimates from \cite{DNY2}.} 
Bringmann proved almost sure local (and global) well-posedness
of \eqref{KG99} 
with the Gibbs measure initial data (essentially corresponding
to the random initial data $(u_0^\o, u_1^\o)$ in \eqref{series}
with $\al =1$),
provided that $\be > 0$.
In the context of the renormalized cubic NLW on $\T^3$, 
it seems possible to adapt  the methodology developed in \cite{Bring2}
and extend Theorem \ref{THM:3}
to $\al > 1$.
When $\al = 1$ (i.e.~\eqref{KG99} with $\be = 0$), 
the argument in \cite{Bring2}  breaks
down in various places and thus further novels ideas are needed to treat the case
$\al = 1$.
We also mention a recent work~\cite{DNY2} by Deng, Nahmod, and Yue,
where they introduced the theory of random tensors
in studying the random data Cauchy theory 
for the nonlinear Schr\"odinger equations.
While this theory is fairly general, as it is pointed out in 
 \cite[Remark 1.6 and Subsection 4.4]{Bring2}, 
 there are some technical challenges in extending
 the theory in \cite{DNY2} to the wave case.

\end{remark}

\subsection{Factorization of the ill-posed solution map}
\label{SUBSEC:fac}

In the following, let us consider initial data of the form:
\begin{align}
(u, \dt u) |_{t = 0} = (w_0, w_1) + (u_0^\o, u_1^\o),
\label{IV1}
\end{align}

\noi
where $(w_0, w_1)$ is a given pair of deterministic functions in $\H^\frac{1}{2}(\T^3)$
and $(u_0^\o, u_1^\o)$ is the random initial data given  in \eqref{series}.
Recall that the random initial data in \eqref{IV1}
belongs almost surely to $\H^{\min(\frac12, s)}(\T^3)$
for $s < \al - \frac 32$.
When $\al > 2$,  the deterministic local well-posedness in $\H^\frac{1}{2}(\T^3)$
yields
a continuous solution map\footnote {Here, the local well-posedness time
is indeed random but we simply write it as $T$.  The same comment applies in the following.}
\begin{align*}
\Phi: 
(w_0, w_1) + (u_0^\o, u_1^\o)
 \in \H^\frac{1}{2}(\T^3) \longmapsto 
( u, \dt u)  \in C([-T, T]; \H^\frac{1}{2}(\T^3)).
 \end{align*}

\noi
On the other hand, when $\al \leq 2$, 
the random initial data in \eqref{IV1} does not belong to $\H^\frac{1}{2}(\T^3)$.
In particular, the ill-posedness results in \cite{Xia, OOTz, FOk}
show that, given any  $(w_0, w_1) \in \H^\frac{1}{2}(\T^3)$, 
 the solution map $\Phi$ is  almost surely discontinuous. 

For $\frac 32 < \al \leq 2$, 
the proof of Theorem \ref{THM:BT}, presented  in \cite{BT-JEMS},   based on the first order expansion~\eqref{Z1a}
yields the following factorization of the ill-posed solution map $\Phi$:
\begin{align}
\begin{split}
(w_0, w_1) + (u_0^\o, u_1^\o)
& \longmapsto
(w_0, w_1, z_1) \stackrel{\Psi_1}{\longmapsto} 
(v, \dt v)  \in C([-T, T]; \H^\frac{1}{2}(\T^3))\\
& \longmapsto
u = z_1 + v  \in C([-T, T]; H^{s_1}(\T^3)), 
\end{split}
\label{IV3}
\end{align}

\noi
where 
 $z_1$ is the solution to the linear  equation \eqref{Z1}
with the random initial data $(u_0^\o, u_1^\o)$ in \eqref{series}
and $s_1 < \al - \frac 32$.
Here, we view 
the first map in \eqref{IV3}  as a  lift map, 
where we use stochastic analysis to construct  an enhanced data set $(w_0, w_1, z_1)$, 
and the second map $\Psi_1$ is the deterministic solution map to the following perturbed NLW:
\begin{align*}
\begin{cases}
\L v + (v+ z_1)^3 = 0\\
(v, \dt v)|_{ t= 0} = (w_0, w_1),
\end{cases}
\end{align*}

\noi
where we view $(w_0, w_1, z_1)$ as an {\it enhanced data set}.\footnote {In particular, we view $z_1$ 
as a given deterministic space-time distribution of some specified regularity.}
Furthermore, the deterministic map $\Psi_1 : (w_0, w_1, z_1)
\mapsto (v, \dt v)$
is continuous from 
\[\mathcal{X}_1^{s_1}(T) : = \H^\frac{1}{2}(\T^3) \times C([-T, T]; W^{s_1, \infty}(\T^3))\]

\noi
to $C([-T, T]; \H^\frac{1}{2}(\T^3))$.

\begin{remark}\rm
In \cite{BT-JEMS},  using a conditional probability, 
Burq and the third author
introduced the notion of probabilistic continuity
and showed that the map: 
$(w_0, w_1) + (u_0^\o, u_1^\o)
\mapsto u$ in~\eqref{IV3}
is indeed probabilistically continuous
when $\frac 32 < \al \leq 2$.
It would be of interest to investigate
if such probabilistic continuity
also holds
for lower values of $\al$.

\end{remark}

For $\frac 43 < \al \leq \frac 32$, 
 the first order expansion \eqref{Z1a}
 along with renormalization
yields the following factorization of the ill-posed solution map $\Phi$:
\begin{align}
\begin{split}
(w_0, w_1) + (u_0^\o, u_1^\o)
& \longmapsto
(w_0, w_1, z_1, Z_2, Z_3) \stackrel{\Psi_2}{\longmapsto} 
(v, \dt v)  \in C([-T, T]; \H^\frac{1}{2}(\T^3))\\
& \longmapsto
u = z_1 + v  \in C([-T, T]; H^{s_1}(\T^3)), 
\end{split}
\label{IV4}
\end{align}

\noi
where 
$Z_2$ and $Z_3$ are the limits of $Z_{2, N}$ and $Z_{3, N}$ in \eqref{L1a}.
With $s_j$, $j = 1, 2, 3$, as in \eqref{S1} and \eqref{S2}, 
 the second map $\Psi_2$ is the deterministic continuous  map, 
sending an enhanced data set $(w_0, w_1, z_1, Z_2, Z_3)$ 
in 
\[\mathcal{X}_2^{s_1, s_2, s_3}(T) : = \H^\frac{1}{2}(\T^3) \times 
\prod_{j = 1}^3C([-T, T]; W^{s_j, \infty}(\T^3))
\]

\noi
to a solution $(v, \dt v)\in C([-T, T]; \H^\frac{1}{2}(\T^3))$ to the following perturbed NLW:
\begin{equation*}
\begin{cases}
\L v+ v^3 + 
3z_{1} v^2+
3Z_2 v
+ Z_3=0\\
(v, \dt v ) |_{t=0}=(w_0, w_1).
\end{cases}
\end{equation*}

For $\frac {13}{10} < \al \leq \frac 43$, 
the proof of Theorem \ref{THM:3}  based on the second order expansion \eqref{Z5}
yields the following factorization of the ill-posed solution map $\Phi$:
\begin{align}
\begin{split}
(w_0, w_1) + (u_0^\o, u_1^\o)
& \longmapsto
(w_0, w_1, z_1, Z_2, z_2) \stackrel{\Psi_3}{\longmapsto} 
(w, \dt w)  \in C([-T, T]; \H^\frac{1}{2}(\T^3))\\
& \longmapsto
u = z_1 +z_2 + w  \in C([-T, T]; H^{s_1}(\T^3)),
\end{split}
\label{IV6}
\end{align}

\noi
where  $z_2$ is the limit of $z_{2, N}$ defined in \eqref{Z4}.
Here, with $s_ 4 = s_3 + 1 < 3(\al - \frac 32) + 1$, 
 the second map $\Psi_3$ is the deterministic continuous  map,  
sending an enhanced data set $(w_0, w_1, z_1, Z_2, z_2)$ 
in 
\[\mathcal{X}_3^{s_1, s_2, s_4}(T) : = \H^\frac{1}{2}(\T^3) \times 
\prod_{j \in\{1, 2, 4\} }C([-T, T]; W^{s_j, \infty}(\T^3))
\]

\noi
to a solution $(w, \dt w)\in C([-T, T]; \H^\frac{1}{2}(\T^3))$ to the following perturbed NLW:
 \begin{equation}
\begin{cases}
\L w+ (w + z_{2})^3
+ 3z_{1} (w + z_{2})^2
+ 3Z_{2}  w + 3Z_{2} z_2=0\\
(w, \dt w ) |_{t=0}=(w_0, w_1).
\end{cases}
\label{IV7}
\end{equation}

Lastly, let us discuss the case $\frac 54 < \al \leq \frac {13}{10}$.
In this case, the product $Z_2 z_2$ in \eqref{IV7}
can not be defined by deterministic paradifferential calculus
and thus we need to define
$Z_5$ as 
a limit of $Z_{5, N}$ in \eqref{L4}.
Then, 
the proof of Theorem \ref{THM:3}  based on the second order expansion~\eqref{Z5}
yields the following factorization of the ill-posed solution map $\Phi$:
\begin{align}
\begin{split}
(w_0, w_1) + (u_0^\o, u_1^\o)
& \longmapsto
(w_0, w_1, z_1, Z_2, z_2, Z_5) \stackrel{\Psi_4}{\longmapsto} 
(w, \dt w)  \in C([-T, T]; \H^\frac{1}{2}(\T^3))\\
& \longmapsto
u = z_1 +z_2 + w  \in C([-T, T]; H^{s_1}(\T^3)).
\end{split}
\label{IV8}
\end{align}

\noi
With $s_5 < 2(\al - \frac 32)$, 
the second map $\Psi_4$ is the deterministic continuous  map, 
sending an enhanced data set $(w_0, w_1, z_1, Z_2, z_2, Z_5)$ 
in 
\begin{align*}
\mathcal{X}_4^{s_1, s_2, s_4, s_5}(T) : = \H^\frac{1}{2}(\T^3) \times 
\prod_{j \in \{1, 2, 4, 5\}} C([-T, T]; W^{s_j, \infty}(\T^3))
\end{align*}

\noi
to a solution $(w, \dt w) \in C([-T, T]; \H^\frac{1}{2}(\T^3))$ to the following perturbed NLW:
 \begin{equation}
\begin{cases}
\L w+ (w + z_{2})^3
+ 3z_{1} (w + z_{2})^2
+ 3Z_{2}  w + Z_5=0\\
(w, \dt w ) |_{t=0}=(w_0, w_1).
\end{cases}
\label{IV9}
\end{equation}

We point out that the last decomposition \eqref{IV8} with \eqref{IV9}
can also be used to study the cases $\frac 43 < \al \leq \frac 32$ and  $\frac{13}{10} < \al \leq \frac 43$.
For simplicity of the presentation, 
we only discuss the last decomposition~\eqref{IV8} with \eqref{IV9}
in this paper, 
while the previous decompositions~\eqref{IV4} and~\eqref{IV6}
provide simpler arguments
when $\frac{13}{10} < \al \leq \frac 32$.

In all the cases mentioned above, 
we decompose the ill-posed solution map $\Phi$
into 
\begin{itemize}
\item[(i)] the first step, constructing enhanced data sets by  stochastic analysis
and 

\item [(ii)] the second step, where purely deterministic analysis is performed
in constructing a continuous map $\Psi_j$ on enhanced data sets,
solving perturbed NLW equations.

\end{itemize}

\noi
Such decompositions of ill-posed solution maps
also appear in studying rough differential equations via the rough path theory~\cite{Lyons, FH}
and 
singular stochastic parabolic PDEs~\cite{H, GIP}.

\begin{remark}\label{REM:uniq2}\rm
By the use of stochastic analysis, 
the terms
$z_1$, $z_2$, $Z_2$, and $Z_5$
are defined as the unique limits
of their truncated versions.
Furthermore, 
by deterministic analysis, we prove that 
a solution $w$ to \eqref{IV9}
is pathwise unique in an appropriate class (see the space $X_T$ defined in \eqref{X1}).
Therefore, under the decomposition  $u = z_1 +z_2 + w$, 
the uniqueness of $u$ claimed in Theorem \ref{THM:3}
refers to 
(i) the uniqueness of $z_1$ and $z_2$ 
as the limits of $z_{1, N}$ and $z_{2, N}$
and (ii) the uniqueness of $w$ as a solution to \eqref{IV9}.

\end{remark}

\begin{remark} \rm
Given $j \in \N_0 := \N\cup\{0\}$,  let $\P_{j}$
 be the (non-homogeneous) Littlewood-Paley projector
 onto the (spatial) frequencies $\{n \in \Z^3: |n|\sim 2^j\}$
 such that 
 \[ f = \sum_{j = 0}^\infty \P_jf.\]
 
 \noi
Given two functions $f$ and $g$ on $\T^3$
of regularities $s_1$ and $s_2$, 
we 
have the following paraproduct decomposition of the product $fg$
due to Bony~\cite{Bony}:
\begin{align}
fg & \hspace*{0.3mm}
= f\pl g + f \pe g + f \pg g\notag \\
& := \sum_{j < k-2} \P_{j} f \, \P_k g
+ \sum_{|j - k|  \leq 2} \P_{j} f\,  \P_k g
+ \sum_{k < j-2} \P_{j} f\,  \P_k g.
\label{para1}
\end{align}

\noi
The first term 
$f\pl g$ (and the third term $f\pg g$) is called the paraproduct of $g$ by $f$
(the paraproduct of $f$ by $g$, respectively)
and it is always well defined as a distribution
of regularity $\min(s_2, s_1+ s_2)$.
On the other hand, 
the resonant product $f \pe g$ is well defined in general 
only if $s_1 + s_2 > 0$.
See Lemma \ref{LEM:para} below.

Let $\frac 54 < \al \leq \frac {13}{10}$.
In this case, the sum of the regularities $s_2 < 2(\al -\frac 32)$
and $s_4 < 3(\al - \frac 32) + 1$ of $Z_2$ and $z_2$
is non-positive and thus we can not make sense of the product $Z_2 z_2$
by deterministic paradifferential calculus.
As we pointed out above, however, 
the paraproducts 
$Z_2 \pl z_2$ and $Z_2 \pg z_2$ are well-defined distributions.
Hence, it suffices to 
 define $Z^{\pe}_5$ as a suitable limit of the resonant products
$ Z_{2, N} \pe z_{2, N}$
in order to pass $3 Z_{2, N} z_{2, N}$ to the limit
\[ Z_5 = 
3 Z_2 \pl z_2 + 3Z_5^{\pe} + 3Z_2 \pg z_2 .\]

\noi
This shows that 
 we can in fact replace the enhanced data set $(w_0, w_1, z_1, Z_2, z_2, Z_5) $
in~\eqref{IV8} 
and $Z_5$ in~\eqref{IV9}
by 
$(w_0, w_1, z_1, Z_2, z_2, Z_5^{\pe})$
and $3Z_2 \pl z_2 + 3Z_5^{\pe} + 3Z_2 \pg z_2 $, respectively.
See also the proof of Proposition~\ref{PROP:Z4} and Remark~\ref{REM:res} below.

\end{remark}

\subsection{NLW without renormalization in negative Sobolev spaces}

We conclude this introduction by discussing 
a new instability phenomenon for NLW \eqref{KG} (that is, without renormalization) in negative Sobolev spaces.
This phenomenon is closely related to the so-called {\it triviality}
in the study of stochastic PDEs
\cite{russo4, HRW}.
See Remark \ref{REM:trivial}.

Fix  a deterministic pair $(w_0, w_1) \in \H^\frac{3}{4}(\T^3)$.
In the following, we  study the (un-renormalized)  NLW \eqref{KG}
with initial data of the form:
\begin{align*}
(u, \dt u )|_{t = 0}  = (w_0, w_1) + (0, u_1^\o), 
\end{align*}

\noi
where $u_1^\o$ is the random distribution given by \eqref{series}.
We consider  this problem by studying 
the following truncated problem.
Given $N \in \N$, 
let $u_N$ be the solution to the (un-renormalized)
NLW~\eqref{KG}
with the following initial data:
\begin{align*}
 (u_N, \dt u_N)|_{t= 0} 
= (w_0, w_1) + 
(\wt u_{0, N}^\o, \wt u_{1, N}^\o). 
\end{align*}

\noi
Here, 
$(\wt u_{0, N}^\o, \wt u_{1, N}^\o)$ denotes the truncated random initial data 
given by 
\begin{equation}\label{series_N2}
\wt u_{0, N}^\o(x) = \sum_{|n|\leq N } \frac{g_n(\o)}{\nbn \jb{n}^{\al-1}}e^{in\cdot x}
\qquad\text{and}\qquad 
\wt u_{1, N}^\o(x) = \sum_{|n|\leq N} \frac{h_n(\o)}{\jb{n}^{\al-1}}e^{in\cdot x},
\end{equation}

\noi
where $\{g_n\}_{n \in \Z^3}$ and $\{h_n\}_{n \in \Z^3}$
are as in \eqref{series}
and 
\begin{align}
 \nbn = \sqrt{C_N + |n|^2}
\label{N1}
\end{align}

\noi
for some suitable choice of a divergent constant $C_N > 0$.
Our goal  is to study the asymptotic behavior of $u_N$ as $N \to \infty$.

Given $N\in \N$, define the linear Klein-Gordon operator $\L_N$ by 
setting
\begin{align}
\L_N :=\partial_t^2 -\Delta +C_N.
\label{lin2}
\end{align}

\noi
Then, 
$u_N$ satisfies
the following equation:
\begin{equation}\label{KG2a}
\begin{cases}
\L_N  u_N+u_N^3 - C_N u_N=0\\
(u_N, \dt u_N)|_{t = 0} = (w_0, w_1) + (\wt u^\o_{0, N}, \wt u^\o_{1, N}).
\end{cases}
\end{equation}

\noi
We denote by $\wt z_{1, N}$ the solution to the following linear Klein-Gordon equation:
\begin{align}
\L_N \wt z_{1,N}= 0
\label{T1}
\end{align}

\noi
with 
the truncated random initial data $(\wt u_{0, N}^\o, \wt v_{0, N}^\o)$ in~\eqref{series_N2}.
Then, we have
\begin{align}
\wt z_{1,N}(t,x,\o)
= \sum_{|n|\leq N }\frac{\cos (t \nbn)} {\nbn \jb{n}^{\al -1}}  g_n(\o)e^{in\cdot x}
+ \sum_{|n|\leq N }\frac{\sin (t \nbn)} {\nbn \jb{n}^{\al-1}}  h_n(\o)e^{in\cdot x}.
\label{T2}
\end{align}

\noi
In particular, for each fixed $(t, x) \in \R\times \T^3$, 
$\wt z_{1,N}(t,x)$ is a mean-zero Gaussian random variable with variance:
\begin{align}
\begin{split}
\wt \s_N : \! & = \E\big[(\wt z_{1, N}(t, x))^2\big]\\
& = \sum_{|n|\leq N }\frac{(\cos (t \nbn))^2 } {\nbn^2 \jb{n}^{2(\al-1)}} 
+ \sum_{|n|\leq N }\frac{(\sin (t \nbn))^2} {\nbn^2 \jb{n}^{2(\al-1)}}\\
& = 
\sum_{|n|\leq N }\frac{1} {\nbn^2 \jb{n}^{2(\al-1)}}.
 \end{split}
\label{T3}
\end{align}

\noi
In view of \eqref{KG2a},
we implicitly define $C_N >0$ by 
\begin{align}
\begin{split}
 C_N  & = 3\wt \s_N 
= 3 \sum_{|n|\leq N }\frac{1} {\nbn^2 \jb{n}^{2(\al-1)}}\\
&  = 3 \sum_{|n|\leq N }\frac{1} {(C_N + |n|^2) \jb{n}^{2(\al-1)}}
\end{split}
\label{CN2}
\end{align}

\noi
such that the subtraction of $ C_N u_N$  in \eqref{KG2a} corresponds to 
(artificial) renormalization of the cubic nonlinearity $u_N^3$.
In Lemma \ref{LEM:CN} below, 
we show that for each $N \in \N$, 
there exists unique $C_N \geq 1$
whose 
asymptotic behavior of $C_N$ as $N \to \infty$
is given by 
\begin{align*}
C_N 
\sim \begin{cases}
\log N, & \text{for } \al = \frac 32,\\
N^{3 - 2\al}, & \text{for } 1\leq \al < \frac 32,\\
\end{cases}
\end{align*}

\noi
for all sufficiently large $N \gg1 $.
In particular, $C_N \to \infty$ as $N \to \infty$
and thus we see that 
  $(\wt u_{0, N}^\o, \wt u_{1, N}^\o)$ 
in  \eqref{series_N2} 
almost surely
converges to $(0, u_1^\o)$
in a suitable topology.

We are now ready to state an instability result 
for the (un-renormalized) NLW \eqref{KG} in negative Sobolev spaces.

\begin{maintheorem}\label{THM:Tri}
Let $\frac 54 < \al < \frac{3}{2}$ and 
 $(w_0, w_1) \in \H^\frac{3}{4}(\T^3)$.
By setting $C_N$ by \eqref{CN2}, 
there exist small  $T_1>0$ 
and positive constants $C$, $c$, $\kappa$
such that for every $T\in (0,T_1]$, 
 there exists a set $\O_T$ of complemental probability smaller than $C\exp(-c/T^\kk)$ 
 such that  if we denote by   $\{ u_N\}_{N\in \N}$ 
the  smooth global   solutions to the defocusing cubic NLW \eqref{KG}
with the random initial data
\begin{align}
 (u_N, \dt u_N)|_{t= 0} 
= (w_0, w_1) + 
(\wt u_{0, N}^\o, \wt u_{1, N}^\o), 
\label{T3a}
\end{align}

\noi
where   $(\wt u_{0, N}^\o, \wt u_{1, N}^\o)$ 
is given 
by  \eqref{series_N2},
 then for every $\o\in \O_T$,
 the sequence $\{u_N\}_{N\in \N}$  converges
 to $0$ as space-time distributions on $[-T, T]\times \T^3$ 
  as $N \to \infty$.

When $  \al = \frac{3}{2}$, the same result holds
but only along a subsequence $\{N_k\}_{k \in \N}$.
Namely, 
there exists an almost surely positive random time $T_\o > 0$ such that 
 the sequence $\{u_{N_k}\}_{k \in \N}$  converges
 to~$0$ as space-time distributions on $[-T_\o, T_\o]\times \T^3$ 
  as $j \to \infty$
  \textup{(}in the sense described above\textup{)}.

\end{maintheorem}

The proof of Theorem \ref{THM:Tri} is based on the reformulation \eqref{KG2a}
and an adaptation of the argument employed in proving Theorem \ref{THM:3}.

We point out that 
the instability  stated in Theorem \ref{THM:Tri}
is due to the lack of  renormalization (in negative regularity).
Indeed, let us briefly discuss the situation
when a proper renormalization is applied.
Consider  the following renormalized NLW:
\begin{equation}
\partial_t^2 u_N-\Delta u_N+u_N^3 - 3\wt \s_N u_N=0
\label{KG5}
\end{equation}

\noi
with the random initial data in \eqref{T3a}.
First, note that 
the initial data in \eqref{T3a}
gives rise to an enhanced data set 
\[\Xi_N = (w_0, w_1, \wz_{1, N}, \wZ_{2,N},  \wz_{2, N},
\wZ_{5, N}),\]

\noi 
where $\wZ_{2,N}$ and $\wZ_{5, N}$
are defined by
\begin{align*}
 \wZ_{2,N}  & :=(\wz_{1,N})^2-\ws_{N}
 \qquad \text{and}\qquad 
 \wZ_{5,N}   :  =\big\{(\wz_{1,N})^2-\ws_{N}\big\} \wz_{2,N}
\end{align*}

\noi
and  $\wz_{2, N}$ is the solution to 
\begin{align*}
\begin{cases}
\L_N \wz_{2,N}+\big\{(\wz_{1,N})^3-3\wt \s_N \wz_{1,N}\big\}=0\\
(z_{2,N}, \dt z_{2, N} ) |_{t=0}=(0, 0).
\end{cases}
\end{align*}

\noi
In Section \ref{SEC:tri}, 
we show that
$\Xi_N$ converges almost surely\footnote{Only along a subsequence $\{N_k\}_{k\in \N}$
when $\al = \frac 32$.} to the limiting enhanced data set
\[\Xi = (w_0, w_1, \wz_{1}, \wZ_{2}, \wz_2, 
\wZ_{5}),\]

\noi
emanating from the initial data
$(w_0, w_1) + (0, u_1^\o)$.
Then, by   slightly modifying  the proof of Theorem \ref{THM:3}, 
we can show that 
the solutions $u_N$ to \eqref{KG5} 
converges to 
 some non-trivial limiting distribution $u = \wz_{1} + \wz_2 + w$, 
where $w$ is the solution to 
 \begin{equation*}
\begin{cases}
\L w+ (w + \wz_{2})^3
+ 3\wz_{1} (w + \wz_{2})^2
+ 3\wZ_{2}  w + \wZ_5=0\\
(w, \dt w ) |_{t=0}=(w_0, w_1).
\end{cases}
\end{equation*}

\noi
Here, we see that $u \not \equiv 0$ 
since the non-zero linear solution $\wz_1$  with initial data $(0, u_1^\o)$
does not belong to $H^{\al - \frac 32 }(\T^3)$ (for a fixed time)
while $z_2 + w \in H^{\al- \frac 32}(\T^3)$ almost surely.
This shows 
the instability result stated in Theorem \ref{THM:Tri}
is peculiar to the case without renormalization
when we work in negative regularities.

\begin{remark}\label{REM:trivial}\rm
The instability result in Theorem \ref{THM:Tri} essentially 
corresponds to triviality results in the study of stochastic PDEs,
where the dynamics without renormalization
trivializes (either to the linear dynamics or the trivial dynamics, i.e.~$u\equiv 0$)
as regularization on a singular random forcing is removed.
See, for example, \cite{russo4, russo5, HRW, OOR, ORSW}.
In particular, our proof of Theorem \ref{THM:Tri} is inspired by the argument in \cite{HRW}
due to Hairer, Ryser, and Weber
for the two-dimensional stochastic nonlinear heat equation.
In the context of the random data Cauchy theory, 
Theorem \ref{THM:Tri} is the first result on triviality 
without renormalization.

In the context of stochastic nonlinear wave equations, 
Albeverio, Haba,  and Russo  \cite{russo4}
studied a triviality issue for the two-dimensional stochastic NLW:
\begin{equation}
\label{SNLW3}
 \dt^2 u - \Dl  u  + f(u) =  \xi, 
\end{equation}

\noi
where 
  $\xi$ is the space-time white noise
  and 
$f$ is a bounded smooth function.
Roughly speaking, they showed that 
solutions to~\eqref{SNLW3} with regularized noises
tend to that to the stochastic linear wave equation:
\[ \dt^2 u - \Dl  u   =  \xi. \]

\noi
Note that
a power-type nonlinearity (such as the  cubic nonlinearity $u^3$)
does not belong to the class of nonlinearities considered in \cite{russo4}.
Furthermore, 
the analysis in~\cite{russo4} was carried out in the framework of Colombeau generalized functions,
and as such, their solution does not a priori belong to 
$C([0, T]; H^{-\eps}(\T^2))$.
In fact, it is not clear if their generalized function represents an actual distribution.
We refer the interested readers to Remark 1.5 in \cite{HairerS} by Hairer and Shen, 
commenting on the work 
\cite{russo5} by 
Albeverio, Haba,  and Russo
on stochastic nonlinear heat equations.
We emphasize that 
our proof of Theorem \ref{THM:Tri}
is based on an adaptation of the solution theory 
for Theorem~\ref{THM:3}.
In particular, for each $N \in \N$, 
we construct the solution~$u_N$ to 
the defocusing cubic renormalized NLW \eqref{KG5}
with the random initial data
in \eqref{T3a}
in the natural space:
$C([0, T]; H^{s}(\T^3))$, 
$s<\al-\frac 32$, 
with a uniform bound in $N \in \N$.
We also point out that a small adaptation of 
the proof of Theorem \ref{THM:Tri}
(as in \cite{OOR})
yield a triviality result
for the following stochastic damped NLW
with the defocusing cubic nonlinearity on~$\T^3$:
\[ \dt^2 u + \dt u -\Dl u + u^3 = \jb{\nb}^{-\al} \xi\]

\noi
for $\frac 54 < \al\le \frac 32$.

After the appearance
of the current paper, 
following  the idea of our triviality result (Theorem \ref{THM:Tri}), 
Okamoto, Robert, and the first author
\cite{OOR}
proved triviality for 
the two-dimensional stochastic 
damped NLW
with the defocusing cubic nonlinearity.
The argument in \cite{OOR} is based on 
an adaptation of the recent solution
theory of the two-dimensional stochastic (damped)
nonlinear wave equations in~\cite{GKO, GKOT}.
We also mention 
a recent work~\cite{ORSW}
by Robert, Sosoe, Y.\,Wang, and the first author
on a triviality result
for  the two-dimensional stochastic wave equation with the sine nonlinearity:
$f(u) = \sin (\be u)$, $\be \in \R\setminus\{0\}$.
While the sine nonlinearity belongs to the class of nonlinearities
considered in~\cite{russo4}, 
the triviality result in \cite{ORSW} is established in the natural class
$C([0, T]; H^{-\eps}(\T^2))$.

In the context of nonlinear Schr\"odinger type equations, 
such instability results without renormalization in negative Sobolev spaces
are known
even {\it deterministically}; see  \cite{GO, OW}.
See also \cite{Ch1, Ch2}
for a similar instability result on the complex-valued mKdV equation
in the deterministic setting.

\end{remark}

\begin{remark}\rm

When $\al = \frac 32$, 
Theorem \ref{THM:Tri} 
yields the almost sure convergence of a subsequence $\{u_{N_k}\}_{k \in \N}$
to $0$ (on a random time interval).
By changing the mode of convergence 
and the related topology, 
it is possible to obtain convergence of the full sequence 
$\{u_{N}\}_{N\in \N}$,  even when $\al = \frac 32$.
More precisely, by slightly modifying the proof of 
Theorem \ref{THM:Tri}, 
we can show that, when $\al = \frac 32$,  
$\{u_{N}\}_{N\in \N}$ 
converges 
 in probability 
to the trivial solution $0$
in $H^{-\eps}([-T_\o, T_\o]; H^{-\eps}(\T^3))$
as $N \to \infty$.
See 
\cite{OOR}
for details of the proof
in the two-dimensional stochastic setting.

\end{remark}

\begin{remark}\rm
In the discussion above, 
we needed to consider the random data $(\wt u_{0, N}^\o, \wt u_{1, N}^\o)$ in~\eqref{series_N2}
in place of $(u^\o_{0, N}, u^\o_{1, N})$ in~\eqref{series_N}
such that $C_N$ can be chosen to be  time independent.
Note that the distribution of  $(\wt u_{0, N}^\o, \wt u_{1, N}^\o)$ in \eqref{series_N2}
is precisely an invariant measure
for the linear dynamics: $\L_N u = 0$.

\end{remark}

\begin{remark}\rm
While the local-in-time results in Theorems~\ref{THM:BT} and \ref{THM:3} 
also holds in the focusing case, 
the proof of Theorem \ref{THM:Tri} only holds for the defocusing case.
In the focusing case, 
we expect some undesirable behavior for solutions to the (un-renormalized) cubic NLW in negative Sobolev spaces
but with a different mechanism.

\end{remark}

\subsection{Organization of the paper}
The remaining part of this manuscript is organized as follows. 
In the next section, 
we state deterministic and stochastic tools needed for our analysis.
In Sections \ref{SEC:sto1} and  \ref{SEC:sto2}, 
we study regularity and convergence properties
of the stochastic terms from Subsection \ref{SUBSEC:outline}.
In Section \ref{SEC:LWP}, we then use the deterministic Strichartz theory
to study the equation \eqref{L3} for $w_N$ and present the proof of  Theorem \ref{THM:3}.
In Section~\ref{SEC:tri}, 
by modifying the analysis from the previous sections, 
we prove Theorem \ref{THM:Tri}.

\section{Tools from deterministic and stochastic analysis}

\subsection{Basic function spaces and paraproducts}

We  define the $L^p$-based Sobolev space $W^{s, p}(\T^3)$
by the norm:
\begin{align*}
\| f \|_{W^{s, p}} = \big\| \F^{-1} (\jb{n}^s \ft f(n))\big\|_{L^p}
\end{align*}

\noi
with the standard modification when $p = \infty$.
When $p = 2$, we have $H^s(\T^3) = W^{s, 2}(\T^3)$.

Next, we recall the regularity properties
of paraproducts and resonant products, viewed as bilinear maps.
For this purpose,  
it is convenient to use the   Besov spaces $B^s_{p, q}(\T^3)$
defined by the norm:
\begin{equation*}
\| u \|_{B^s_{p,q}} = \Big\| 2^{s j} \| \P_{j} u \|_{L^p_x} \Big\|_{\l^q_j(\N_0)}.
\end{equation*}

\noi
Note that  $H^s(\T^3) = B^s_{2,2}(\T^3)$.

\begin{lemma}\label{LEM:para}
\textup{(i) (paraproduct and resonant product estimates)}
Let $s_1, s_2 \in \R$ and $1 \leq p, p_1, p_2, q \leq \infty$ such that 
$\frac{1}{p} = \frac 1{p_1} + \frac 1{p_2}$.
Then, we have 
\begin{align}
\| f\pl g \|_{B^{s_2}_{p, q}} \les 
\|f \|_{L^{p_1}} 
\|  g \|_{B^{s_2}_{p_2, q}}.  
\label{para2a}
\end{align}

\noi
When $s_1 < 0$, we have
\begin{align}
\| f\pl g \|_{B^{s_1 + s_2}_{p, q}} \les 
\|f \|_{B^{s_1 }_{p_1, q}} 
\|  g \|_{B^{s_2}_{p_2, q}}.  
\label{para2}
\end{align}

\noi
When $s_1 + s_2 > 0$, we have
\begin{align}
\| f\pe g \|_{B^{s_1 + s_2}_{p, q}} \les 
\|f \|_{B^{s_1 }_{p_1, q}} 
\|  g \|_{B^{s_2}_{p_2, q}}  .
\label{para3}
\end{align}

\noi
\textup{(ii)}
Let $s_1 <  s_2 < s_3$ and $1\leq p, q \leq \infty$.
Then, we have 
\begin{align} 
\| u \|_{W^{s_1, p}}
\les 
\| u \|_{B^{s_2}_{p,q}} 
&\les \| u \|_{W^{s_3, p}}.
\label{embed}
\end{align}

\end{lemma}

The product estimates \eqref{para2a},  \eqref{para2},  and \eqref{para3}
follow easily from the definition \eqref{para1} of the paraproduct 
and the resonant product.
See \cite{BCD, MW2} for details of the proofs in the non-periodic case
(which can be easily extended to the current periodic setting).
The embeddings~\eqref{embed}
 follow from the $\l^{q}$-summability 
of $\big\{2^{(s_k - s_{k+1})j}\big\}_{j \in \N_0}$ for $s_k < s_{k+1}$, $k = 1, 2$, 
and the uniform boundedness of the Littlewood-Paley projector $\P_j$.
Thanks to \eqref{embed}, 
we can apply the product estimates \eqref{para2a}, \eqref{para2}, and \eqref{para3}
in the Sobolev space setting (with a slight loss of regularity).

\subsection{Product estimates, an interpolation inequality, and Strichartz estimates}

For $s\in\R$, we set
$\jb{\nb}^{s}:=(1-\Delta)^{\frac{s}{2}}$.
Then, we have the following standard product estimates.
See \cite{GKO} for their proofs.

\begin{lemma}\label{LEM:gko}
Let $0\leq s\leq 1$.

\noi
\textup{(i)}
Let  $1<p_j,q_j,r<\infty$, $j=1,2$ such that $\frac{1}{r}=\frac{1}{p_j}+\frac{1}{q_j}$.
Then, we have 
$$
\|\jb{\nb}^s(fg)\|_{L^r(\T^3)}\lesssim\| \jb{\nb}^s f\|_{L^{p_1}(\T^3)} \|g\|_{L^{q_1}(\T^3)}+ \|f\|_{L^{p_2}(\T^3)} 
\|  \jb{\nb}^s g\|_{L^{q_2}(\T^3)}.
$$

\noi
\textup{(ii)}
Let $1<p,q,r<\infty$ such that $s \geq   3\big(\frac{1}{p}+\frac{1}{q}-\frac{1}{r}\big)$.
Then, we have
$$
\|\jb{\nb}^{-s}(fg)\|_{L^r(\T^3)}
\lesssim\| \jb{\nb}^{-s} f\|_{L^{p}(\T^3)} \| \jb{\nb}^{s} g\|_{L^{q}(\T^3)} .
$$
\end{lemma}

Note that
while  Lemma \ref{LEM:gko} (ii) 
was shown only for 
$s =   3\big(\frac{1}{p}+\frac{1}{q}-\frac{1}{r}\big)$
in \cite{GKO}, 
the general case
$s \geq    3\big(\frac{1}{p}+\frac{1}{q}-\frac{1}{r}\big)$
follows from a straightforward modification.

\medskip

Next, we state an interpolation inequality.
This lemma allows us to reduce an estimate on the $L^\infty$-norm in time
to that with the $L^q$-norm in time for some finite $q$.

\begin{lemma}\label{LEM:BTT}
Let $T > 0$ and $1 \leq q, r \leq \infty$.
Suppose that $s_1, s_2, s_3 \in \R$ satisfy $s_2 \leq s_1$
and 
\[ q( s_1 - s_3) > s_1 - s_2.\]

\noi
Then, we have
\begin{align*}
\| u \|_{L^\infty([-T, T]; W^{s_3, r}(\T^3))}
\les 
\| u \|_{L^q([-T, T]; W^{s_1, r}(\T^3))}^{1-\frac 1q}
\| u \|_{W^{1, q}([-T, T]; W^{s_2, r}(\T^3))}^\frac{1}{q}.
\end{align*}

\noi
Here, the $W^{1, q}([-T, T];W^{s, r}(\T^3))$-norm is defined by 
\[\|f \|_{W^{1, q}([-T, T];W^{s, r}(\T^3))} 
=\|f \|_{L^q([-T, T];W^{s, r}(\T^3))} + \|\dt f \|_{L^q([-T, T];W^{s, r}(\T^3))} . \]
\end{lemma}

The proof of Lemma \ref{LEM:BTT} follows from duality in $x$
and
Gagliardo-Nirenberg's inequality in $t$ along with 
standard analysis based on (spatial) Littlewood-Paley decompositions.
See the proofs of  Lemmas 3.2 and 3.3 in \cite{BTT1} for the $r = 2$ case.
The proof for the general case follows
from a straightforward modification.

\medskip

We now recall the Strichartz estimates.
Let $\L$ be the Klein-Gordon operator in \eqref{lin0}.
We use 
$\L^{-1}= (\dt^2 - \Dl + 1)^{-1}$ to denote the Duhamel integral operator, corresponding to the forward fundamental solution to the Klein-Gordon equation:
\begin{align}
\L^{-1} F(t) := \int_0^t \frac{\sin((t-t')\jb{\nb})}{\jb{\nb}} F(t') dt'.
\label{lin1}
\end{align}

\noi
Namely,  $u:=\L^{-1}(F)$ is the solution to the following nonhomogeneous linear equation:
\begin{equation*}
\begin{cases}
\L u= F\\ 
(u,\partial_t u)|_{t=0}=(0,0).
\end{cases}
\end{equation*}

\noi
The most basic regularity property of $\L^{-1}$ is 
the energy estimate:
\begin{equation}\label{basic_reg}
\|\L^{-1}(F)\|_{L^\infty([-T,T];H^{s}(\T^3))}\lesssim \|F\|_{L^1([-T,T];H^{s-1}(\T^3))}.
\end{equation}

\noi
The Strichartz estimates  are important extensions of \eqref{basic_reg}
and  have been studied extensively by many
mathematicians.  See \cite{GV, LS, KeelTao}
in the context of the wave equation on $\R^d$.
Thanks to the finite speed of propagation, 
the Strichartz estimates on $\T^3$ follow from the corresponding
estimates on $\R^3$, locally in time.
We now  state 
the Strichartz estimates
which are 
relevant for the analysis in this paper. 
We refer to \cite{tz-cime} for a detailed proof.

\begin{lemma}\label{LEM:Str}

Let $0 < T \leq 1$.	
Then, the following estimate holds:
\begin{align}\label{strich}
\begin{split}
\|\L^{-1}(F)\|_{L^4([-T,T]\times\T^3)}
& +
\|\L^{-1}(F)\|_{L^\infty([-T,T];H^{\frac{1}{2}}(\T^3))}
\\
& \lesssim 
\min\Big(
\|F\|_{L^1([-T,T];H^{-\frac{1}{2}}(\T^3))}
,
\|F\|_{L^{\frac{4}{3}}([-T,T]\times\T^3)}
\Big).
\end{split}
\end{align}

\end{lemma}

For $T>0$, we denote by $X_T$ the closed subspace of $C([-T,T];H^\frac{1}{2}(\T^3))$ endowed with the norm:
\begin{align}
\|u\|_{X_T}=\|u\|_{L^\infty([-T,T];H^\frac{1}{2}(\T^3))}+\|u\|_{L^4([-T,T]\times \T^3)}\,.
\label{X1}
\end{align}

\noi
In the following, we use shorthand notations such as
$L^q_T L^r_x := L^q([-T, T]; L^r(\T^3))$.

\subsection{On discrete convolutions}

Next, we recall the following basic lemma on a discrete convolution.

\begin{lemma}\label{LEM:SUM}
\textup{(i)}
Let $d \geq 1$ and $\al, \be \in \R$ satisfy
\[ \al+ \be > d  \qquad \text{and}\qquad \al, \be < d.\]
\noi
Then, we have
\[
 \sum_{n = n_1 + n_2} \frac{1}{\jb{n_1}^{\al} \jb{n_2}^{\be}}
\les \jb{n}^{d -\al -\be}\]

\noi
for any $n \in \Z^d$.

\smallskip

\noi
\textup{(ii)}
Let $d \geq 1$ and $\al, \be \in \R$ satisfy $\al+ \be > d$.
\noi
Then, we have
\[
 \sum_{\substack{n = n_1 + n_2\\|n_1|\sim|n_2|}} \frac{1}{\jb{n_1}^{\al} \jb{n_2}^{\be}}
\les \jb{n}^{d - \al - \be}\]

\noi
for any $n \in \Z^d$.

\end{lemma}

Namely, in the resonant case (ii), we do not have the restriction $\al, \be < d$.
Lemma \ref{LEM:SUM} follows
from elementary  computations.
See, for example,  Lemmas 4.1 and 4.2 in \cite{MWX} for the proof.

\subsection{Wiener chaos estimate}

Lastly, we  recall the following Wiener chaos estimate
\cite[Theorem~I.22]{Simon}.
See also \cite[Proposition~2.4]{TTz}.

\begin{lemma}\label{LEM:hyp}
 Let $\{ g_n\}_{n \in \N }$ be 
 a sequence of  independent standard real-valued Gaussian random variables.
Given  $k \in \mathbb{N}$, 
let $\{P_j\}_{j \in \N}$ be a sequence of monomials in 
$\bar  g = \{ g_n\}_{n \in \N }$ of  degree at most $k$,
 namely, 
$P_j = P_j (\bar g)$ is of the form
$P_j =  c_j \prod_{i = 1}^{k_j} g_{n_{i}} $
with $k_j \leq k$ and $n_{1}, \dots, n_{k_j} \in \N$.
Then, for $p \geq 2$, we have
\begin{equation*}
 \bigg\|\sum_{j \in \N} P_j(\bar g) \bigg\|_{L^p(\O)} \leq (p-1)^\frac{k}{2} \bigg\|\sum_{j \in \N} P_j(\bar g) \bigg\|_{L^2(\O)}.
 \end{equation*}

\end{lemma}

This lemma is a direct corollary to the
  hypercontractivity of the Ornstein-Uhlenbeck
semigroup due to Nelson \cite{Nelson2}.
Note that in the definition of $P_j$ above, 
we may have $n_{i} = n_{\l}$ for $i \ne \l$.
Namely, we do not impose independence of the factors $g_{n_{i}}$ of $P_j$
in Lemma~\ref{LEM:hyp}.
In the following, we apply Lemma \ref{LEM:hyp}
to multilinear terms involving 
$\{ g_n \}_{n \in \Z^3}$ and  $\{ h_n \}_{n \in \Z^3}$
 in~\eqref{series}
by first expanding $g_n$ and $h_n$ into their 
real and imaginary parts.

\section{On the random free evolution and its renormalized powers}
\label{SEC:sto1}

Recall from \eqref{series_N} and \eqref{Z1} that 
$z_{1,N}(t,x,\o)$ denotes the solution to the linear Klein-Gordon equation: 
\begin{align*}
(\partial_t^2 -\Delta +1)z_{1,N}(t,x,\o)=0
\end{align*}

\noi
with the truncated random initial data:
$$
z_{1,N}(0,x,\o)= \sum_{|n|\leq N } \frac{g_n(\o)}{\jb{n}^{\al}}e^{in\cdot x}
\qquad \text{and} \qquad 
\partial_t z_{1,N}(0,x,\o) = \sum_{|n|\leq N} \frac{h_n(\o)}{\jb{n}^{\al-1}}e^{in\cdot x}, 
$$

\noi
where
$\{ g_n \}_{n \in \Z^3}$ and  $\{ h_n \}_{n \in \Z^3}$ are as in \eqref{series}.
Given $t \in \R$, define $g^t_n(\o)$ by 
\begin{align}
g^t_n(\o):=\cos(t\langle n\rangle)\, g_{n}(\o)+ \sin(t\langle n\rangle)\,h_n(\o).
\label{G1}
\end{align}

\noi
Then, we have 
\begin{align*}
\begin{split}
z_{1,N}(t,x,\o)
& =  \cos(t \jb{\nb})\Big(z_{1,N}(0,x,\o)\Big)
+ \frac{\sin(t \jb{\nb})}{\jb{\nb}}\Big(\dt z_{1,N}(0,x,\o)\Big)\\
& = \sum_{|n|\leq N } \frac{g^t_n(\o)}{\jb{n}^{\al}}e^{in\cdot x}.
\end{split}
\end{align*}

\noi
Using the definitions  of the Gaussian random variables $\{g_n\}_{n \in \Z^3}$ and $\{h_n\}_{n \in \Z^3}$, 
we see that $\{g_n^t\}_{n \in \Z^3}$ defined in \eqref{G1}
forms a family of 
independent standard   complex-valued  Gaussian random variables  
 conditioned that\footnote {In particular, 
$g_0^t$ is real-valued.}  $g_n^t=\overline{g^t_{-n}}$.
Then, the 
renormalization constant~$\s_N$ 
defined in~\eqref{sig}
is computed as
\begin{align}
\begin{split}
\s_N 
& = 
\E\Big[ \big(z_{1,N}(t,x,\o)\big)^2 \Big]
 =  \sum_{|n|\leq N}
\frac{\E\big[|g^t_n(\o)|^2\big]}{\jb{n}^{2\al}}
\\
& =   \sum_{|n|\leq N} 
\frac{1}{\jb{n}^{2\al}}
\sim \begin{cases}
\log N, & \text{for } \al = \frac 32, \rule[-3mm]{0pt}{0pt}\\
N^{3-2\al}, & \text{for } \al < \frac 32,
\end{cases}
\end{split}	
\label{CN}
\end{align}

\noi
which tends to $\infty$ as $N \to \infty$.

\begin{remark}\label{REM:sig}\rm	
From the definitions  of the Gaussian random variables $g_n$ and $h_n$
and their rotational invariance, we  see that 
\[\text{Law}(z_{1, N}(t, x)) = \text{Law}(z_{1, N}(0,0))\]

\noi
for any $(t, x) \in \R \times \T^3$.
This also explains the independence of $\s_N$ from $t$ and $x$.
\end{remark}

We now define the sequences $\{Z_{j,N}\}_{N\in \N}$, $j = 1, 2, 3$, 
by
\begin{align}
Z_{1,N}:=z_{1,N},
\quad 
Z_{2,N}:=(z_{1,N})^2-\s_{N},
\quad \text{and}\quad  Z_{3,N}:=(z_{1,N})^3- 3\s_{N} z_{1,N} .
\label{G3}
\end{align}

\noi
The main goal of this section is to prove the following proposition
on the regularity and convergence properties
of the stochastic terms $Z_{1, N}$, $Z_{2, N}$, and $Z_{3, N}$.

\begin{proposition}\label{PROP:ran1}
Let $1< \al \leq \frac 32$ and set
\begin{align}
s_1 < \al - \tfrac 32, \quad
s_2 < 2(\al - \tfrac 32), 
\quad 
\text{and}
\quad 
s_3< 3(\al - \tfrac 32).
\label{reg1}
\end{align}

\noi
Fix $j = 1, 2,$ or $3$. Then, given any $T>0$, $Z_{j, N}$ converges
almost surely 
to some limit $Z_j$ in $C([-T, T]; W^{s_j, \infty}(\T^3))$
as $N \to \infty$.
Moreover, 
given $2\leq q < \infty$, 
there exist 
positive constants $C$, $c$, $\kappa$, $\ta$ 
such that for every $T> 0$, 
 there exists a  set $\O_T$ of complemental probability smaller than $C\exp(-c/T^\kk)$ 
 with the following properties; 
 given $\eps > 0$, there exists 
$N_0  = N_0(T, \eps) \in \N$
such that 
\begin{equation}\label{conv2}
\big\|Z_{j,N}\big\|_{L^q([-T,T];W^{s_j, \infty}(\T^3))}\leq T^\ta
\end{equation}

\noi
and
\begin{equation}\label{conv1}
\big\|Z_{j,M}-Z_{j,N}\big\|_{C([-T,T];W^{s_j, \infty}(\T^3))} <\eps
\end{equation}

\noi
for any  $\o\in\O_T$
and any $M\geq N\geq N_0$,
where we allow $N = \infty$ with the understanding that $Z_{j, \infty} = Z_j$.

\end{proposition}

We split  the proof of this proposition into several parts.
We first present preliminary  lemmas and then prove Proposition \ref{PROP:ran1}
at the end of this section.

\begin{lemma}\label{LEM:nel}
Let $1 < \al \leq \frac 32$ and $s_j$, $j = 1, 2, 3$, satisfy \eqref{reg1}.
Then, 
given  $2\leq q < \infty$ and $2 \leq r\leq \infty$, 
there exists $\delta>0$ such that 
the following estimates hold for   $j=1,2,3$:
\begin{align}
\big\|\jb{\nb}^{s_j}Z_{j,N}\big\|_{L^p(\O; L^q([-T,T];L^r(\T^3)))}
& \leq C  T^{\frac{1}{q}} p^{\frac{j}{2}}, 
\label{hyper1}
\\
\Big\|\jb{\nb}^{s_j}\big(Z_{j,M}-Z_{j,N}\big)\Big\|_{L^p(\O; L^q([-T,T];L^r(\T^3)))}
& \leq C N^{-\delta} T^{\frac{1}{q}} p^{\frac{j}{2}}  ,
\label{hyper2}
\end{align}

\noi
for   any $ M\geq N \geq 1$, $T > 0$, and any finite $p \geq 1$, 
where the constant $C$ is independent of $M,N,T,p$.
\end{lemma}

\begin{proof} 
In the following, we only prove the difference estimate \eqref{hyper2}
since the first estimate~\eqref{hyper1} follows
in a similar manner.

When $r = \infty$, 
we can apply the Sobolev embedding theorem
and reduce the $r = \infty$ case to the case of large but finite $r$ 
at the expense of a slight loss of spatial derivative.
This, however, does not cause an issue
since  the conditions on $s_j$ are open.
Hence, we assume $r < \infty$ in the following.

Let  $p\geq \max(q,r)$.
Since
$$
\jb{\nb}^{s_1}Z_{1,N}
=
\sum_{|n|\leq N } \frac{g^t_n(\o)}{\jb{n}^{\al-s_1}}e^{in\cdot x}, 
$$

\noi
we see that 
$ \jb{\nb}^{s_1}\big(Z_{1,N}-Z_{1,M}\big)(t, x)$ is
a mean-zero Gaussian random variable for fixed $t$ and $x$.
In particular, there exists a universal constant $C>0$ such that 
\begin{align}
\Big\| \jb{\nb}^{s_1}\big(Z_{1,M}-Z_{1,N}\big)(t, x)
\Big\|_{L^p(\O)}
\leq C p^\frac{1}{2}\Big\| \jb{\nb}^{s_1}\big(Z_{1,M}-Z_{1,N}\big)(t, x)
\Big\|_{L^2(\O)}.
\label{G4}
\end{align}

\noi
Then, 
it follows from 
 Minkowski's integral inequality and \eqref{G4}
that 
\begin{align}
\begin{split}
\bigg\|\Big\|\jb{\nb}^{s_1}
 \big(Z_{1,M}- & Z_{1,N}\big)\Big\|_{L^q_TL^r_x}\bigg\|_{L^p(\O)}
  \leq
\bigg\|\Big\| \jb{\nb}^{s_1}\big(Z_{1,M}-Z_{1,N}\big)(t, x)
\Big\|_{L^p(\O)}
\bigg\|_{L^q_TL^r_x}\\
& 
\leq C T^\frac{1}{q}p^\frac{1}{2} \bigg(\sum_{N<|n|\leq M } \frac{1}{\jb{n}^{2(\al-s_1)}}\bigg)^\frac{1}{2}
\leq C N^{ - \dl }T^{\frac{1}{q}} p^\frac{1}{2}
\end{split}
\label{G4a}
\end{align}

\noi
for some $\dl>0$ under the regularity assumption \eqref{reg1}.
This proves \eqref{hyper2} for $j = 1$.

Next, we  turn to  the $j=2$ case.
Let us  write 
\begin{align}
\jb{\nb}^{s_2}Z_{2,N}= \I_N+\II_N,
\label{G5}
\end{align}
where
$$
\I_N (t, x):=
\sum_{\substack{|n_1|\leq N,  |n_2|\leq N \\ n_1\neq -n_2}}
 \frac{g^t_{n_1}(\o)g^t_{n_2}(\o) } {\jb {n_1+n_2}^{-s_2} \jb{n_1}^{\al}\jb{n_2}^{\al}} e^{i(n_1+n_2)\cdot x}
$$
and 
$$
\II_N(t, x):=
\sum_{|n|\leq N}
\langle n\rangle^{-2\al}
\Big(
|g^t_n(\o)|^2 - \E\big[|g^t_n|^2\big]
\Big)
= \sum_{|n|\leq N}
\langle n\rangle^{-2\al}
\Big(
|g^t_n(\o)|^2 -1
\Big).
$$

\noi
Fix $(t, x) \in \R \times \T^3$.
By using the independence of $\{g_n^t\}_{n \in \Ld}$  with $\Ld$ as in \eqref{index}
and Lemma~\ref{LEM:SUM}, we have\footnote {Strictly speaking, 
in applying Lemma \ref{LEM:SUM} when $\al = \frac 32$, we need to replace $2\al$ in the exponent
by $2\al - \eps$ for some small $\eps>0$.
This, however, does not affect the outcome since the condition
on $s_2$ is open.
The same comment applies to~\eqref{G8b} below.}
\begin{align}
\begin{split}
\|\I_M(t, x)-\I_N(t, x)\|^2_{L^2(\O)}
& \lesssim 
\sum_{
\substack
{
|n_1|\leq M,  |n_2|\leq M 
\\
\max(|n_1|,|n_2|)>N
}}
 \frac{1}
 {
 \jb {n_1+n_2}^{-2s_2}
 \jb{n_1}^{2\al}
  \jb{n_2}^{2\al}
 }\\
& \leq CN^{-\delta},
\end{split}
\label{G6}
\end{align}

\noi
for some $\delta>0$,  with $C$ independent of $M\geq  N\geq 1$ and $(t, x)\in \R\times \T^3$,  
provided that $4\al- 2s_2>6$.
Namely, 
$s_2 < 2(\al - \frac 32)$.
Similarly, by using the independence of $\big\{|g^t_n(\o)|^2 -1\big\}_{n \in \Ld}$, 
we have
\begin{align}
\|\II_M(t, x)-\II_N(t, x)\|^2_{L^2(\O)}\lesssim 
\sum_{N<|n|\leq M}
\frac{1}{\jb{n}^{4\al}}
\leq CN^{-\delta}
\label{G7}
\end{align}

\noi
for some $\delta>0$, with $C$ independent of $M\geq  N\geq 1$ and $(t, x) \in \R\times \T^3$, provided $4\al>3$, 
which is guaranteed by the assumption $\al > 1$.
Therefore, from \eqref{G5}, \eqref{G6}, and \eqref{G7}, 
we obtain 
$$
\Big\|\jb{\nb}^{s_2}\big(Z_{2,M}-Z_{2,N}\big)(t, x)\Big\|_{L^2(\O)}\leq C N^{-\delta}
$$

\noi
for some $\delta>0$,  with a constant $C$ independent of $M\geq  N \geq 1$ and $(t, x)\in \R\times \T^3$.  
By the Wiener chaos estimate (Lemma \ref{LEM:hyp}), 
we then obtain 
\begin{align}
\Big\|
\jb{\nb}^{s_2}\big(Z_{2,M}-Z_{2,N}\big)(t, x)\Big\|_{L^p(\O)}
\leq C N^{-\delta} p
\label{G8}
\end{align}

\noi
for any finite $p \geq 2$.
Then, arguing as in \eqref{G4a} with Minkowski's integral inequality, 
the estimate  \eqref{hyper2} for $j=2$ follows from \eqref{G8}.

Let us finally turn to \eqref{hyper2} for $j=3$.   Write 
$$
\jb{\nb}^{s_3}Z_{3,N}=\III_N+\IV_N,
$$

\noi
where 
$$
\III_N(t, x):=
\sum_{\substack{|n_j|\leq N,\, j = 1, 2, 3\\
(n_1+n_2)(n_1+n_3)(n_2+n_3)\neq 0}}
 \frac{g^t_{n_1}(\o)g^t_{n_2}(\o)  g^t_{n_3}(\o) }
 {
 \jb {n_1+n_2+n_3}^{-s_3}
 \jb{n_1}^{\al}
  \jb{n_2}^{\al}
  \jb{n_3}^{\al}
 }
 e^{i(n_1+n_2+n_3)\cdot x}
$$

\noi
and by the inclusion-exclusion principle
\begin{align*}
\IV_N(t, x)
:\! & =
3 \sum_{|n|\leq N} 
\frac{|g^t_n(\o)|^2 -\E\big[|g^t_n|^2\big]}
{\langle n\rangle^{2\al}}
\sum_{|m|\leq N}
\frac{g_m^t(\o)}
 {\langle m \rangle^{\al - s_3}}
e^{im\cdot x}\\
& \hphantom{X}
- 3 
 \sum_{|n|\leq N} 
\frac{|g^t_n(\o)|^2g_n^t(\o)}
{\langle n\rangle^{3\al-s_3}}
e^{in\cdot x}
+ |g_0^t(\o)|^2 g_0^t(\o).
\end{align*}

\noi
Proceeding as above with Lemma \ref{LEM:SUM}, we have 
\begin{align}
\begin{split}
\|\III_M(t, x)& -\III_N(t, x)\|^2_{L^2(\O)}\\
& \lesssim 
\sum_{
\substack
{
|n_j|\leq M, \, j = 1, 2, 3\\
\max(|n_1|,|n_2|, |n_3|)>N
}}
 \frac{1}
 {
 \jb {n_1+n_2+n_3}^{-2s_3}
 \jb{n_1}^{2\al}
  \jb{n_2}^{2\al}
  \jb{n_3}^{2\al}
 }
\leq CN^{-\delta}
\end{split}
\label{G8b}
\end{align}

\noi
for some $\delta>0$, with $C$ independent of $M\geq N\geq 1$ and $(t, x)\in \R \times \T^3$, 
provided $6\al-2s_3>9$ and $\al > 1$. 
See Remark \ref{REM:cube}.
Namely, $s_3<3(\al - \frac{3}{2})$ and~$\al > 1$.
Then, by  the Wiener chaos estimate (Lemma~\ref{LEM:hyp}), we obtain
\begin{align}
\|\III_M(t, x)-\III_N(t, x)\|_{L^p(\O)}\leq C N^{-\delta} p^{\frac{3}{2}}
\label{G8a}
\end{align}

\noi
for any finite $p \geq 2$.

Let us now estimate $\IV_N$. 
By  
Lemma \ref{LEM:hyp} and 
H\"older's inequality, we have
\begin{align*}
\big\|\IV_N(t, x)\big\|_{L^p(\O)}
& \les  
p^{\frac{3}{2}}
 \sum_{|n|\leq N} 
\frac{1}
{\langle n\rangle^{3\al-s_3}}  \\
& \hphantom{X}
+
\bigg \|
\sum_{|n|\leq N} 
\frac{|g^t_n(\o)|^2 -\E\big[|g^t_n|^2\big]}
{\langle n\rangle^{2\al}}
\bigg\|_{L^{2p}(\O)}
\bigg\|
\sum_{|m|\leq N}
\frac{g_m^t(\o)}
 {\langle m \rangle^{\al-s_3}}
e^{im\cdot x}
\bigg\|_{L^{2p}(\O)}.
\end{align*}

\noi
The first sum on the right-hand side is convergent if $3\al - s_3 > 3$.
Note that this condition is guaranteed under \eqref{reg1}.
Both factors in the second term on the right hand-side 
can 
be treated by the arguments presented above.
We therefore have the bounds:
\begin{align}
\bigg \|
\sum_{|n|\leq N} 
\frac{|g^t_n(\o)|^2 -\E\big[|g^t_n|^2\big]}
{\langle n\rangle^{2\al}}
\bigg\|_{L^{2p}(\O)}
\leq C p
\label{G9}
\end{align}

\noi
and 
\begin{align}
\bigg\|
\sum_{|m|\leq N}
\frac{g_m^t(\o)}
 {\langle m \rangle^{\al-s_3}}
e^{im\cdot x}
\bigg\|_{L^{2p}(\O)}
\leq C p^{\frac{1}{2}}
\label{G10}
\end{align}

\noi
\noi
for any finite $p \geq 2$, 
provided that $4\al > 3$ for \eqref{G9} and 
$2\al - 2s_3 > 3$ for \eqref{G10}.
Note that the second condition is guaranteed under \eqref{reg1}
with $\al \leq \frac 32$.
Then, by applying the Wiener chaos estimate (Lemma \ref{LEM:hyp}), 
this leads to
$$
\|\IV_N(t, x)\|_{L^p(\O)}\leq C p^{\frac{3}{2}}.
$$

\noi
A similar argument yields
\begin{align}
\|\IV_M(t, x)-\IV_{N}(t, x)\|_{L^p(\O)}\leq C N^{-\delta} p^{\frac{3}{2}} 
\label{G11}
\end{align}

\noi
for some  $\delta > 0$. 
Then, arguing as in \eqref{G4a} with Minkowski's integral inequality, 
the estimate~\eqref{hyper2} for $j=3$ follows from \eqref{G8a} and \eqref{G11}.
This completes the proof of Lemma~\ref{LEM:nel}.
\end{proof}

Thanks to Lemma~\ref{LEM:nel}, we already know that    the sequences 
$\{Z_{j,N}\}_{N\in \N}$,  $j=1,2,3$, converge in $L^p(\O;L^q([-T,T];W^{s_j, r}(\T^3)))$ to some limits $Z_{j}$.
It turns out that  the quantitative 
properties \eqref{hyper2} of the convergence allow us to upgrade these convergences to almost sure convergences. 
See the proof of Proposition \ref{PROP:ran1} below.
In order to obtain convergence in $C([-T, T]; W^{s_j, r}(\T^3))$, 
however,  we need to establish a difference estimate at two different times.
The following lemma will be useful in this context.

\begin{lemma}\label{LEM:hhhlll}
Let $k\geq 1$ be an integer. Then, we can write 
\begin{equation}\label{developp}
\prod_{j=1}^k g^t_{n_j}-\prod_{j=1}^k g^\tau_{n_j}
=
\sum _{\l} c_{\l}(t,\tau,n_1,\cdots,n_k)\prod_{j=1}^k g^*_{n_j},
\end{equation}

\noi
where $g^*_{n_j}$ is either $g_{n_j}$ or $h_{n_j}$ and the sum in $\l$ runs over all such possibilities.
Furthermore, given any  $\delta>0$, there exists $C_\delta>0$ such that 
\begin{align}
|c_{\l}(t,\tau,n_1,\cdots,n_k)|\leq C_{\delta}|t-\tau|^\delta \sum_{j=1}^k \langle n_j\rangle^\delta\,.
\label{typ0}
\end{align}
\end{lemma}

\begin{proof}
From the definition \eqref{G1}
of $g_n^t$, a typical term in the sum defining the right-hand side of \eqref{developp} is given by
\begin{equation}\label{typ}
\Big(\prod_{j=1}^kH_j(t\jb{n_j} ) -\prod_{j=1}^kH_j(\tau \jb{n_j} )\Big)\prod_{j=1}^k g^*_{n_j},
\end{equation}

\noi
where $H_j(t\jb{n_j}) = \cos(t\jb{n_j})$ 
(with $g^*_{n_j} = g_{n_j}$) or 
$\sin(t\jb{n_j})$ (with $g^*_{n_j} = h_{n_j}$).
By the mean value theorem and the boundedness of $H_j$, 
we have 
\begin{align}
\bigg|
\prod_{j=1}^kH_j(t\jb{n_j} ) -\prod_{j=1}^kH_j(\tau\jb{n_j} )
\bigg|
\lesssim 
|t-\tau|\sum_{j=1}^k \jb{n_j}.
\label{typ1}
\end{align}

\noi
We also have the trivial bound 
\begin{align}
\bigg|
\prod_{j=1}^kH_j(t\jb{n_j} ) -\prod_{j=1}^kH_j(\tau\jb{n_j} )
\bigg|\leq 2.
\label{typ2}
\end{align}

\noi
By interpolating \eqref{typ1} and \eqref{typ2}, 
we conclude that  \eqref{typ} satisfies the claimed bound~\eqref{typ0}. 
This completes the proof of Lemma~\ref{LEM:hhhlll}.
\end{proof}

In view of  Lemma~\ref{LEM:hhhlll}, 
a  slight modification of the proof of Lemma~\ref{LEM:nel} yields the following statement. 

\begin{lemma}\label{LEM:nel_pak}
Let $1 < \al \leq \frac 32$ and $s_j$ satisfies \eqref{reg1},  $j = 1, 2, 3$.
Then, 
given  $2 \leq r\leq \infty$, 
there exists $\delta>0$ such that 
the following estimates hold for   $j=1,2,3$:
\begin{align}
\big\|\jb{\nb}^{s_j}\delta_h Z_{j,N}(t)\big\|_{L^p(\O; L^r(\T^3))}
& \leq 
C  p^{\frac{j}{2}}|h|^{\delta}  ,
\label{hyper3}
\\
\Big\|\jb{\nb}^{s_j}\big(\delta_h Z_{j,M}(t)-\delta_{h}Z_{j,N}(t)\big)\Big\|_{L^p(\O; L^r(\T^3))}
& \leq 
C N^{-\delta} p^{\frac{j}{2}}|h|^{\delta}  ,
\label{hyper4}
\end{align}

\noi
for any $M\geq N \geq 1$,   $t\in [-T,T]$, and $h \in \R$ 
such that $t+ h \in [-T, T]$, 
where the constant $C$ is independent of $M,N,T,p, t$, 
and $h$.
Here, $\dl_h$ denotes the difference operator defined by 
\[\delta_h Z_{j,N}(t)=Z_{j,N}(t+h)-Z_{j,N}(t).\]

\end{lemma}

In handling 
the renormalized pieces, 
we also need the following identity, 
which follows directly from  \eqref{G1}:
\begin{align*}
\Big(|g_n^t|^2 & - \E\big[|g_n^t|^2\big]\Big)
 - \Big(  |g_n^\tau|^2  - \E\big[|g_n^\tau|^2\big]\Big)\\
& = \Big( \cos^2 (t\jb{n}) - \cos^2 (\tau\jb{n}) \Big) \Big( |g_n|^2 - 1\Big)\\
& \hphantom{X}
+ \Big( \sin^2 (t\jb{n}) - \sin^2 (\tau\jb{n}) \Big) \Big( |h_n|^2- 1\Big)\\
& \hphantom{X}
+ 2\Big( \cos (t\jb{n}) \sin (t\jb{n})
-   \cos (\tau\jb{n}) \sin (\tau\jb{n})\Big)
\cdot
\Re (g_n \cj{h_n}).
\end{align*}

\noi
The first two terms on the right-hand side
can be treated exactly as in the renormalized pieces in 
the proof of Lemma~\ref{LEM:nel}, 
while the last term can be handled without any difficulty.

\smallskip

We conclude this section by presenting the proof of 
Proposition \ref{PROP:ran1}.

\begin{proof}[Proof of Proposition \ref{PROP:ran1}]
Fix  $2\leq q < \infty$ and $j = 1, 2,$ or $3$.
Passing to the limit $N\rightarrow\infty$ in~\eqref{hyper1} of Lemma~\ref{LEM:nel}, 
we obtain that the limit $Z_j$ of $Z_{j,N}$ satisfies
$$
\big\| \| Z_j\|_{L^q_T W^{s_j, \infty}_x}\big\|_{L^p(\O)} \leq CT^{\frac{1}{q}}p^{\frac{j}{2}} 
$$

\noi
for any finite $p \geq 1$.
Then, it follows from Chebyshev's inequality\footnote {See for example Lemma 4.5 in \cite{TzBO}
and the proof of Lemma 3/2.2 in \cite{BOP1}.} that there exists 
a set $\O_{T, \infty}^{(1)}$ of complemental probability smaller than $C\exp(-c/T^\frac{2}{jq})$  such that 
\begin{align}
\label{H1}
\big\| Z_j\big\|_{L^q_TW^{s_j, \infty}_x}\leq \frac 12 T^{\frac{1}{2q}}
\end{align}

\noi
for any $\o\in\O_{T, \infty}^{(1)}$.
Similarly, given any $N \in \N$, 
it follows from \eqref{hyper2} (with $M \to \infty$) that 
 there exists 
a set $\O_{T, N}^{(1)}$ of complemental probability smaller than $C\exp(-cN^\frac{2\dl}{j}/T^\frac{2}{jq})$  such that 
\begin{align}
\label{H2}
 \big\|Z_{j}-Z_{j,N}\big\|_{L^q_TW^{s_j, \infty}_x}
&\leq \frac 12 T^{\frac{1}{2q}}
\end{align}

\noi
for any $\o\in\O_{T, N}^{(1)}$.
Combining \eqref{H1} and \eqref{H2}, 
we see that   \eqref{conv2}
holds
for any $\o \in \O_T^{(1)}$ defined by 
\begin{align}
 \O_{T}^{(1)} := \bigcap_{N \in \N \cup \{\infty\}} \O_{T, N}^{(1)}
\label{H3}
 \end{align}

\noi
whose complemental probability is 
smaller than $C\exp(-c/T^\frac{2}{jq})$.

Lemma \ref{LEM:nel} shows that   the sequence
$\{Z_{j,N}\}_{N\in \N}$  converges in 
\[L^p(\O;L^q([-T,T];W^{s_j, \infty}(\T^3)))\] 

\noi
to 
the  limit $Z_{j}$.
A slight modification of the proof of 
 Lemma~\ref{LEM:nel}
 shows that, given $t \in \R$, 
 the sequence 
$\{Z_{j,N}(t)\}_{N\in \N}$
converges to the limit $Z_{j}(t)$ in $L^p(\O; W^{s_j, \infty}(\T^3))$
with the uniform bound:
\begin{align*}
\big\|Z_{j,N}(t)\big\|_{L^p(\O; W^{s_j, \infty})}
& \leq C  p^{\frac{j}{2}}.
\end{align*}

\noi
We first upgrade this convergence to almost sure convergence.
From \eqref{hyper2} in Lemma~\ref{LEM:nel} and 
Chebyshev's inequality, we obtain that 
$$
P\bigg(\o\in \O: \|Z_{j}(t)-Z_{j, N}(t)\|_{W^{s_j, \infty}}\geq\frac{1}{k}\bigg)
\leq e^{-c N^\frac{2\dl}{j} k^{-\frac{2}{j}} }
$$

\noi
for $k \in \N$, 
 where the positive constant $c$ is independent of $k$ and $N$.
 Noting that the right-hand side is summable in $N \in \N$, 
 we can invoke 
 the Borel-Cantelli lemma
 to conclude that  there exists $\O_k$ of full probability such that
  for each $\o\in \O_k$,  there exists $M=M(\o) \geq 1$ such that for any $N\geq M$, 
  we have
$$
 \|Z_{j}(t;\o)-Z_{j, N}(t;\o)\|_{W^{s_j, \infty}}<\frac{1}{k}.
$$

\noi
Now, by setting $\Sigma=\bigcap_{k=1}^\infty\O_k$, 
we see that   $P(\Sigma)=1$ and that, for each $\o\in\Sigma$, 
$Z_{j,N}(t)$ converges to $Z_{j}(t)$ in $W^{s_j, \infty}(\T^3)$. 
Note that the set $\Si$ is dependent on the choice of $t \in \R$.

We now  prove 
that $\{Z_{j, N}\}_{N \in \N}$ converges to $Z_j$ almost surely in
$C([-T,T];W^{s_j, \infty}(\T^3))$.
Fix $t \in [-T, T]$ and  $h \in \R$ (such that $t + h \in [-T, T]$). 
From 
\eqref{hyper3}, \eqref{hyper4}, 
the almost sure convergence of $Z_{j, N}(t)$
to $Z_j(t)$, and the dominated convergence theorem, 
we obtain
\begin{align}
\big\|\delta_h Z_{j}(t)\big\|_{L^p(\O; W^{s_j, \infty})}
& \leq 
C  p^{\frac{j}{2}}|h|^{\delta}  ,
\label{hyper6}
\\
\big\|\delta_h Z_{j}(t)-\delta_{h}Z_{j,N}(t)\big\|_{L^p(\O; W^{s_j, \infty})}
& \leq 
C N^{-\delta} p^{\frac{j}{2}}|h|^{\delta}  
\label{hyper7}
\end{align}

\noi
for any $N \geq 1$.
By choosing $p \gg 1$ sufficiently large such that $p\dl > 1$, 
it follows from  Kolmogorov's continuity criterion \cite{Bass}
applied to \eqref{hyper3} and \eqref{hyper6}
that  $Z_{j, N}$, $N \in \N$,  and $Z_j$
are almost surely continuous with values in 
$W^{s_j, \infty}(\T^3)$.

In the following, we only consider $[0, T]$.
Let $Y_N = Z_j - Z_{j, N}$
and choose  $p \gg 1$ sufficiently large such that $p \dl > 2$.
Then, with $\ta \in \big(0, \dl - \frac 1p\big)$, 
it follows from Chebyshev's inequality and~\eqref{hyper7} that 
\begin{align*}
P\bigg( & \sup_{N \in \N} 
\max_{j = 1, \dots, 2^\l}
N^{\frac{\dl}{2}}
\Big\|Y_N\big(\tfrac{j}{2^\l}T\big) - Y_N\big(\tfrac{j-1}{2^\l}T\big)\Big\|_{W^{s_j, \infty}}
\geq  2^{-\ta\l}
\bigg)\notag\\
& = P\bigg(
\bigcup_{N \in \N}
\bigcup_{j = 1}^{2^\l}
\bigg\{
 \Big\|Y_N\big(\tfrac{j}{2^\l}T\big) - Y_N\big(\tfrac{j-1}{2^\l}T\big)\Big\|_{W^{s_j, \infty}}
\geq N^{-\frac{\dl}{2}} 2^{-\ta\l}\bigg\}
\bigg)\notag\\
& \leq  
\sum_{N = 1}^\infty \sum_{j = 1}^{2^\l}
P\bigg(\Big\|Y_N\big(\tfrac{j}{2^\l}T\big) - Y_N\big(\tfrac{j-1}{2^\l}T\big)\Big\|_{W^{s_j, \infty}}
\geq N^{-\frac{\dl}{2}} 2^{-\ta\l}
\bigg)\notag\\
& \leq  
\sum_{N= 1}^\infty
\sum_{j = 1}^{2^\l}
 N^{\frac{p\dl}{2}} 2^{p\ta \l}
\, \E\bigg[\Big\|Y_N\big(\tfrac{j}{2^\l}T\big) - Y_N\big(\tfrac{j-1}{2^\l}T\big)\Big\|_{W^{s_j, \infty}}^p\bigg]\notag\\
& \leq C(p)  \cdot 2^{( p(\ta  -\dl) + 1)\l}
\sum_{N= 1}^\infty
 N^{- \frac{p\dl}{2}}\\
& \leq C(p)  \cdot 2^{( p(\ta  -\dl) + 1)\l}, 
\end{align*}

\noi
where we used the fact that $p \dl > 2$ in the second to the last step.
Note that $p(\ta  -\dl) + 1 <0.$
Then, summing over $\l \in \N$, we obtain 	
\begin{align*}
\sum_{\l = 0}^\infty P\bigg( 
\sup_{N \in \N}  \max_{j = 1, \dots, 2^\l}
N^{\frac \dl 2}\Big\|Y_N\big(\tfrac{j}{2^\l}T\big) - Y_N\big(\tfrac{j-1}{2^\l}T\big)\Big\|_{W^{s_j, \infty}}
\geq  2^{-\ta\l}
\bigg) < \infty.
\end{align*}

\noi
Hence, by the Borel-Cantelli lemma, 
there exists a set $\wt \Sigma \subset \O$ with $P(\wt \Si) = 1$
such that, for each $\o \in \wt \Si$, we have 
\begin{align*}
\sup_{N\in \N} \max_{j = 1, \dots, 2^\l}
N^{\frac \dl 2} \Big\|Y_N\big(\tfrac{j}{2^\l}T;\o\big) - Y_N\big(\tfrac{j-1}{2^\l}T;\o\big)\Big\|_{W^{s_j, \infty}}
\leq  2^{-\ta\l}
\end{align*}

\noi
for all $\l \geq L = L(\o)$.
This in particular implies that there exists $C = C(\o)>0$ such that 
\begin{align}
\max_{j = 1, \dots, 2^\l}
\Big\|Y_N\big(\tfrac{j}{2^\l}T;\o\big) - Y_N\big(\tfrac{j-1}{2^\l}T;\o\big)\Big\|_{W^{s_j, \infty}}
\leq C(\o) N^{-\frac \dl 2} 2^{-\ta\l}
\label{K7}
\end{align}

\noi
for any $\l \geq 0$, uniformly in $N \in \N$.

Fix $t \in [0, T]$.
By expressing $t$  in the following binary expansion (dilated by $T$):
\begin{align*}
 t = T\sum_{j = 1}^\infty \frac{b_j}{2^j}, 
 \end{align*}

\noi
where $b_j \in \{0, 1\}$,
we set
 $t_\l = T\sum_{j = 1}^\l \frac{b_j}{2^j}$
and $t_0 = 0$.
Then, from \eqref{K7}
along with the continuity of $Y_N$ with values in $W^{s_j, \infty}(\T^3)$, we have
\begin{align}
\begin{split}\
\|Y_N(t; \o)\|_{W^{s_j, \infty}}
& \leq \sum_{\l = 1}^\infty  \|Y_N(t_\l;\o)- Y_N(t_{\l-1};\o)\|_{W^{s_j, \infty}} + \|Y_N(0;\o)\|_{W^{s_j, \infty}}\\
& \leq C(\o) N^{-\frac \dl 2}
\sum_{\l = 1}^\infty
2^{-\ta\l} + \|Y_N(0;\o)\|_{W^{s_j, \infty}}\\
& \leq C'(\o) N^{-\frac \dl 2}
 + \|Y_N(0;\o)\|_{W^{s_j, \infty}}
\end{split}
 \label{K8}
\end{align}

\noi
for each $\o \in \wt \Si$.
Note that the right-hand side of \eqref{K8} is independent of $t\in [0, T]$.
Hence, by taking a supremum in $t \in [0, T]$, we obtain 
\begin{align*}
\|Z_j(\o) - Z_{j, N}(\o)\|_{C([0, T];W^{s_j, \infty})}
& \leq C'(\o) N^{-\frac \dl 2}
 + \|Y_N(0;\o)\|_{W^{s_j, \infty}}\notag\\
 & \too 0
\end{align*}

\noi
as $N \to \infty$.
Here, we used the almost sure convergence of $\{Z_{j, N}(0)\}_{N \in \N}$ 
to $Z_j(0)$
in $W^{s_j, \infty}(\T^3)$.
This proves almost sure convergence of $\{Z_{j, N}\}_{N \in \N}$
in $C([-T, T]; W^{s_j, \infty}(\T^3))$.

Lastly,  it follows from  Egoroff's theorem
that, given $T>0$,
there exists $\O_T^{(2)}$  of complemental probability smaller than $C\exp(-c/T^\kk)$ 
such that 
the estimate \eqref{conv1} holds.
Finally, by setting $\O_T = \O_T^{(1)} \cap \O_T^{(2)}$, 
where $\O_T^{(1)}$ is as in \eqref{H3}, 
we see that both \eqref{conv2} and \eqref{conv1} hold
on $\O_T$.
 This completes the proof of Proposition~\ref{PROP:ran1}.
\end{proof}

\begin{remark}\label{REM:cube}\rm
The restriction $\al > 1$ in Proposition \ref{PROP:ran1}
appears in making sense of the renormalized cubic power $Z_{3}$
and it
reflects the well-known fact that Wick powers of degree $\geq 3$ 
for  the three-dimensional Gaussian free field do not exist.
See for example Section~2.7 in~\cite{EJS}.

\end{remark}

\section{On the second order stochastic term $z_{2,N}$}
\label{SEC:sto2}

We first study the regularity and convergence properties of $z_{2, N}$ defined in \eqref{Z4}.
For notational convenience, we set 
\begin{align}
\begin{split}
Z_{4, N}: = z_{2,N}
& =-\L^{-1}\big((z_{1,N})^3-3\s_N z_{1,N}\big) \\
& =-\L^{-1}Z_{3, N}.
\end{split}
\label{z2}
\end{align}

\noi
As a consequence of Proposition~\ref{PROP:ran1}, we have the following statement.

\begin{proposition}\label{PROP:z2}
Let $1< \al \leq \frac 32$ and set
\begin{align}
s_4<3(\al-\tfrac{3}{2}) + 1.
\label{reg_z2}
\end{align}

\noi	
Then, given any $T>0$, $Z_{4, N}$ converges
almost surely 
to some limit $Z_4$ in $C([-T, T]; W^{s_4, \infty}(\T^3))$
as $N \to \infty$.
Moreover, 
there exist positive constants $C$, $c$, $\kappa$, $\ta$ such that for every $T>0$,
 there exists a  set $\O_T$ of complemental  probability smaller than $C\exp(-c/T^\kk)$ 
such that given $\eps>0$,  there exists $N_0 = N_0(T, \eps) \in \N$ such that  
\begin{align*}
\|Z_{4,N}\|_{C([-T,T];W^{s_4, \infty}(\T^3))} & \leq T^\ta
\end{align*}

\noi
and 
\begin{align*}
\|Z_{4,M}-Z_{4,N}\|_{C([-T,T];W^{s_4, \infty}(\T^3))} &  <\eps
\end{align*}

\noi
for any $\o \in \O_T$ and any $M\geq N \geq 1$, 
where we allow $N = \infty$ with the understanding that $Z_{4, \infty} = Z_4$.
\end{proposition}

\begin{proof}

Given $s_4$ satisfying \eqref{reg_z2}, 
choose $\eps > 0$ sufficiently small such that 
\begin{align}
s_4+2\eps <3(\al-\tfrac{3}{2}) + 1.
\label{ZZ0}
\end{align}

\noi
By Sobolev's inequality, there exists finite $r\gg 1$ such that 
\begin{align}
\|Z_{4, N}\|_{C_TW^{s_4, \infty}_x}
\les \|Z_{4, N}\|_{C_TW^{s_4+\eps, r}_x}.
\label{ZZ1}
\end{align}

\noi
Furthermore, by Lemma \ref{LEM:BTT}, there exists finite $q \gg1 $ such that 
\begin{align}
\begin{split}
\|Z_{4, N}\|_{C_TW^{s_4+\eps, r}_x}
& \leq \| Z_{4, N}\|_{L^q_T  W^{s_4+2\eps, r}_x}^{1 - \frac 1q}
\| Z_{4, N}\|_{W^{1,q} _T  W^{s_4-1, r}_x}^\frac{1}{q}\\
& \les 
\| Z_{4, N}\|_{L^q_T  W^{s_4+2\eps, r}_x}
+  \| \dt Z_{4, N}\|_{L^q_T  W^{s_4-1, r}_x}, 
\end{split}
\label{ZZ2}
\end{align}
	
\noi
where we applied
 Young's inequality in the second step.
From \eqref{z2} with \eqref{lin1}, 
we have 
\begin{align}
\dt Z_{4, N} = - \int_0^t \cos ((t - t') \jb{\nb}) Z_{3, N}(t') dt'.
\label{ZZ3}
\end{align}

\noi
Hence, from \eqref{ZZ1}, \eqref{ZZ2}, and \eqref{ZZ3} with \eqref{z2}, we obtain
\begin{align}
\begin{split}
\|Z_{4, N}\|_{C_TW^{s_4, \infty}_x}
& \les T^{1 - \frac{1}{q}} \sum_{\be \in \{ -1, 1\}}
\big\| F^\be_{N}(t, t') \big\|_{L^q_{t, t'} ([-T, T]^2;   W^{s_4-1+2\eps, r}_x)}, 
\end{split}
\label{ZZ3a}
\end{align}

\noi
where $F^\be_{N}$ is given by 
\begin{align*}
 F^\be_{N}(t, t') =  e^{i \be  (t - t')  \jb{\nb}} Z_{3, N}(t').
\end{align*}

\noi
Fix  $(t, t', x) \in \R^2\times \T$.
Since the propagator $e^{i \be  (t - t')  \jb{\nb}}$ does not affect 
 the computation done for $Z_{3, N}$ in the proof of Lemma \ref{LEM:nel}, 
we obtain 
\begin{align}
\big\| F^\be_{N}(t, t', x)\big\|_{L^p(\O)} \leq C p^\frac{3}{2}, 
\label{ZZ4}
\end{align}

\noi
uniformly in   $(t, t', x) \in \R^2\times \T$.
Therefore, given finite $p \geq \max (q, r)$, 
from~\eqref{ZZ3a}, Minkowski's integral inequality,  and \eqref{ZZ4}, 
we have
\begin{align*}
\begin{split}
\Big\|\|Z_{4, N}\|_{C_TW^{s_4, \infty}_x}\Big\|_{L^p(\O)}
& \les T^{1 - \frac{1}{q}} \sum_{\be \in \{ -1, 1\}}
\Big\| \big\|F^\be_{N}(t, t', x) \big\|_{L^p(\O)} \Big\|_{L^q_{t, t'} ([-T, T]^2;   W^{s_4-1+2\eps, r}_x)}\\
& \les p^\frac{3}{2}T^{1 - \frac{1}{q}}
\end{split}
\end{align*}

\noi
thanks to the regularity restriction \eqref{ZZ0}.
Then, the rest  follows from proceeding
as in the proof of 
 Proposition~\ref{PROP:ran1}
 (in addition to Lemma~\ref{LEM:hhhlll},
  one should take into account the trivial continuity property in $t$ of the time integration 
  in the definition of~$Z_{4,N}$). 
 \end{proof}

We also need to study the following quintic  stochastic term:
\begin{align}
\begin{split}
Z_{5,N}:
\! & =\big\{(z_{1,N})^2-\s_{N}\big\} z_{2,N}\\
& = - \big\{(z_{1,N})^2-\s_{N}\big\}
\cdot \L^{-1}\big((z_{1,N})^3-3\s_N z_{1,N}\big) .
\end{split}
\label{B1}
\end{align}

\noi
We have the following statement.

\begin{proposition}\label{PROP:Z4}
Let $1< \al \leq \frac 32$ and set
\begin{align}
s_5 < \min\big( 5\al - \tfrac {13}{2}, 2(\al - \tfrac 32)\big).
\label{reg2}
\end{align}

\noi
Then, given any $T>0$, $Z_{5, N}$ converges
almost surely 
to some limit $Z_5$ in $C([-T, T]; W^{s_5, \infty}(\T^3))$
as $N \to \infty$.
Moreover, 
there exist positive constants $C$, $c$, $\kappa$, $\ta$ such that for every $T>0$,
 there exists a  set $\O_T$ of complemental  probability smaller than $C\exp(-c/T^\kk)$ 
such that given $\varepsilon>0$,  there exists $N_0 = N_0(T, \eps) \in \N$ such that  
\begin{align*}
\|Z_{5,N}\|_{C([-T,T];W^{s_5, \infty}(\T^3))} & \leq T^\ta
\end{align*}
and 
\begin{align*}
\|Z_{5,M}-Z_{5,N}\|_{C([-T,T];W^{s_5, \infty}(\T^3))} &  <\varepsilon
\end{align*}

\noi
for any $\o \in \O_T$ and any $M\geq N \geq 1$, 
where we allow $N = \infty$ with the understanding that $Z_{5, \infty} = Z_5$.

\end{proposition}

\begin{remark}\rm
When $\al \geq \frac{7}{6}$
(which in particular includes the case $ \al > \frac 54$), 
the regularity condition~\eqref{reg2} reduces to 
$s_5 <  2(\al - \tfrac 32)$.

\end{remark}

\begin{proof}
By the paraproduct decomposition \eqref{para1}, 
we have
\begin{align*}
Z_{5,N}
& =Z_{2,N} z_{2,N}\\
& = Z_{2,N} \pl z_{2,N} + Z_{2,N}\pe z_{2,N}+ Z_{2,N}\pg z_{2,N}.
\end{align*}

\noi
Note that $2(\al - \frac 32) \leq  \min\big(0,  3(\al - \frac32)+1\big)$
for $\al \in (1, \frac 32]$.
Then,  from Lemma~\ref{LEM:para}, we have
\begin{align*}
\| Z_{2,N} \pl z_{2,N}(t)\|_{W^{s_5, \infty}}
\les \| Z_{2,N}(t)\|_{W^{2(\al - \frac 32) - \eps, \infty}} 
\|  z_{2,N}(t)\|_{W^{ 3(\al - \frac32)+1 -\eps, \infty}}
\end{align*}

\noi
for small $\eps > 0$, provided that $s_5$ satisfies
\begin{align}
s_5 < 2(\al - \tfrac 32) +   3(\al - \tfrac32) + 1= 5\al - \tfrac{13}{2}.
\label{B3}
\end{align}

\noi
Similarly,  for $s_5$ satisfying \eqref{B3}, 
 Lemma~\ref{LEM:para} yields
\begin{align*}
\| Z_{2,N} \pg z_{2,N}(t)\|_{W^{s_5, \infty}}
\les \| Z_{2,N}(t)\|_{W^{2(\al - \frac 32) - \eps, \infty}} 
\|  z_{2,N}(t)\|_{W^{3(\al - \frac32) +1-\eps, \infty}}
\end{align*}

\noi
for small $\eps > 0$, provided that 
$ 3(\al - \frac32)+1 -\eps < 0$
namely, ~$\al \leq \frac 76$.
On the other hand, when  $ \al > \frac 76$, 
we see from Proposition \ref{PROP:z2} that $z_{2, N}$ has a spatial positive regularity (for each fixed $t$).
In this case, we have
\begin{align*}
\| Z_{2,N} \pg z_{2,N}(t)\|_{W^{s_5, \infty}}
\les \| Z_{2,N}(t)\|_{W^{2(\al - \frac 32) - \eps, \infty}} 
\|  z_{2,N}(t)\|_{L^\infty}
\end{align*}

\noi
as long as 
\begin{align}
s_5 < 2(\al - \tfrac 32).
\label{B6}
\end{align}

\noi
Note that the condition \eqref{B6} is stronger than \eqref{B3} 
when $\al > \frac 76$.

It remains to study the resonant product $z_{2,N}\pe Z_{2,N}$.
When $\al > \frac{13}{10}$,
we have
\begin{align*}
2(\al - \tfrac 32) +  3(\al - \tfrac32) +1 = 5\al - \tfrac{13}{2} > 0.
\end{align*}	

\noi
and thus Lemma \ref{LEM:para} yields
\begin{align*}
\| Z_{2,N} \pe z_{2,N}(t)\|_{W^{s_5, \infty}}
\les \| Z_{2,N}(t)\|_{W^{2(\al - \frac 32) - \eps, \infty}} 
\|  z_{2,N}(t)\|_{W^{ 3(\al - \frac32) +1 -\eps, \infty}}
\end{align*}
  
\noi
for $s_5$ satisfying \eqref{B3}.
Next, we consider the case $1 < \al \leq \frac{13}{10}$.
Using the independence of $\{g^t_n\}_{n \in \Ld} $, we have
\begin{align*}
\begin{split}
\sup_{N \in \N} \sup_{(t,x)\in \R\times\T^3} 
\E\Big[   |\jb{\nb}^{s_5 } &  (  Z_{2,N}   \pe z_{2,N})(t, x) |^2 \Big]\\
&  \les
\sum_{n\in\Z^3}\langle n\rangle^{2s_5}
\sum_{\substack{n = m_1+m_2\\|m_1|\sim |m_2|}}A(m_1)B(m_2),
\end{split}
\end{align*}

\noi
where $A(m_1)$ and $B(m_2)$ are given by 
$$
A(m_1)=\sum_{m\in\Z^3}\frac{1}{\langle m\rangle^{2\al}\langle m_1-m\rangle^{2\al}}
$$

\noi
and
$$
B(m_2)=\frac{1}{\jb{m_2}^2}\sum_{(n_1,n_2)\in\Z^6}\frac{1}{\langle n_1\rangle^{2\al}\langle n_2\rangle^{2\al}\langle m_2-n_1-n_2\rangle^{2\al}}
$$

In the following, we only consider 
the case $\al < \frac 32$.
We clearly have the bound
$$
A(m_1)\lesssim \frac{1}{\langle m_1\rangle^{4\al-3}},
$$

\noi
provided that $\al>\frac{4}{3}$. 
Similarly, by Lemma \ref{LEM:SUM}, we have 
$$
B(m_2)\lesssim  \frac{1}{\langle m_2\rangle^{6\al-4}}, 
$$

\noi
provided that $\al > 1$.
Hence, we obtain
\begin{align*}
\sum_{\substack{n = m_1+m_2\\|m_1|\sim |m_2|}}\,A(m_1)B(m_2)
\lesssim
\sum_{\substack{m_1\in \Z^3\\|m_1|\sim |n-m_1|}}
  \frac{1}{
\jb{m_1}^{4\al-3}
\jb{n - m_1}^{6\al-4}}
\lesssim\frac{1}{\langle n\rangle^{10\al-10}}, 
\end{align*}

\noi
where we crucially used the resonant restriction $|m_1|\sim |n-m_1|$.
Therefore, we obtain
\begin{align*}
\sup_{N \in \N} \sup_{(t,x)\in [0,T]\times\T^3} 
\E\Big[   |\jb{\nb}^{s_5 }   (  Z_{2,N}   \pe z_{2,N})(t, x) |^2 \Big]
 \les
\sum_{n\in\Z^3}\frac{1}{\jb{n}^{-2s_5 + 10 \al - 10}},
\end{align*}

\noi
where the last sum is convergent, provided that   $s_5$ satisfies \eqref{B3}.
With  this bound in hand, we can proceed as in the proof of 
Proposition \ref{PROP:ran1}
(with Lemmas~\ref{LEM:nel} and~\ref{LEM:nel_pak}).
This completes the proof of Proposition~\ref{PROP:Z4}.
\end{proof}

\begin{remark}\label{REM:res} \rm
(i) When $\al > \frac{13}{10}$, 
we made sense of the resonant product
$Z_{2,N} \pe z_{2,N}$
in a  {\it deterministic} manner.
Namely, we only used the almost sure regularity properties
of $Z_{2, N}$ and $z_{2, N}$ but did not use the random structure of these terms
in making sense of their resonant product.
On the other hand, 
when $1 < \al \leq \frac{13}{10}$, 
the sum of the regularities of 
 $Z_{2, N}$ and $z_{2, N}$ is negative
  and thus
 their resonant product does not make sense in a deterministic manner.
 This requires us to make sense of the resonant product 
$Z_{2,N} \pe z_{2,N}$
via a {\it probabilistic} argument.
Hence, when $1 < \al \leq \frac{13}{10}$, 
 we need to view 
the limit 
$Z_{5}^{\pe} = Z_{2, \infty} \pe z_{2, \infty}$ as part of a predefined enhanced data set, leading
to a different interpretation of the  equation for $w = u - z_1 - z_2$.
See Subsection \ref{SUBSEC:fac} for a further discussion.
Lastly, we point out that the resulting regularity
restriction \eqref{B3} holds for both cases
$\al > \frac{13}{10}$ and $1 < \al \leq \frac{13}{10}$.

\smallskip

\noi
(ii) When $\al = 1$, 
there is a logarithmically  divergent contribution in  taking a limit of $Z_{5, N}$
as $N \to \infty$.
In this case, we need to introduce another renormalization,
eliminating a quartic singularity.
For a related argument in the parabolic setting, 
see \cite{MWX}.

\end{remark}


\section{Proof of Theorem~\ref{THM:3}}
\label{SEC:LWP}

\subsection{Setup} 
Recall that
$
u_{N}=z_{1,N}+z_{2,N}+w_{N},
$
where $w_N$ solves the equation \eqref{L3}. 
In Sections \ref{SEC:sto1} and \ref{SEC:sto2}, 
we already established the necessary regularity and  convergence properties of the sequences 
$\{z_{j,N}\}_{N\in \N}$, $j=1,2$. 
It remains to establish the convergence of the sequence  $\{w_{N}\}_{N\in \N}$. 
This will be done by 
(i) first establishing multilinear estimates via a purely deterministic method
and then
(ii) applying the 
regularity and  convergence properties of
the relevant stochastic terms from Sections \ref{SEC:sto1} and \ref{SEC:sto2}.

With \eqref{G3} and \eqref{B1}, 
we can write the equation \eqref{L3} as 
\begin{equation*}
\begin{cases}
\L w_N+F_0+F_1(w_N)+F_2(w_N)+F_3(w_N)=0\\
(w_N, \dt  w_N)|_{t=0}=(0, 0),
\end{cases}
\end{equation*}

\noi
where the source term\footnote {Namely, the purely stochastic terms independent
of the unknown $w_N$.} is given by
\begin{align*}
F_0=3Z_{5, N}+ 3z_{1,N}(z_{2,N})^2+(z_{2,N})^3,
\end{align*}

\noi
the linear term in $w_N$ is given by 
\begin{align*}
F_{1}(w_N)=3Z_{2, N}w_N+6z_{1,N}  z_{2,N} w_N+3 (z_{2,N})^2 w_{N},
\end{align*}

\noi
and the nonlinear terms in $w_N$ are as follows: 
\begin{align*}
F_2(w_N)=3z_{1,N} (w_N)^2+3 z_{2,N}w_N^2
\qquad \text{and} \qquad F_3(w_N)=w_N^3.
\end{align*}

In the following, we study the Duhamel formulation for $w_N$:
\begin{align}
w_{N}=\L^{-1}(F_0)+\L^{-1}(F_1(w_N))+\L^{-1}(F_2(w_N))+\L^{-1}(F_3(w_N)).
\label{w4a}
\end{align}

\noi
In the next three subsections, 
we first establish estimates for each individual term in the $X_T$-norm defined in \eqref{X1}.
In Subsection~\ref{SUBSEC:end}, 
we then combine these estimates
with the 
regularity and  convergence properties of
the relevant stochastic terms from Sections~\ref{SEC:sto1} and~\ref{SEC:sto2}
and prove 
almost sure convergence of the sequence  $\{w_{N}\}_{N\in \N}$. 
In the following, we fix $0 < T \leq 1$.

\subsection{On the nonlinear terms  in  $w_N$}
\label{SUBSEC:pf2}

By the Strichartz estimate \eqref{strich}, we have
\begin{equation}\label{F3}
\big\|\L^{-1}(F_3(w_N))\big\|_{X_T}\lesssim \|w_N^3\|_{L^\frac{4}{3}_{T, x}}\leq \|w_N\|_{X_T}^3.
\end{equation}

We now turn to the analysis of $\L^{-1}(F_2(w_N))$. 
By \eqref{strich}, we have
\begin{equation}\label{ik1}
\big\|\L^{-1}(F_2(w_N))\big\|_{X_T}
\lesssim \big\|\jb{\nb}^{-\frac{1}{2}}\big(z_{1,N} w_N^2+z_{2,N}w_N^2\big)\big\|_{L^1_T L^2_x}.
\end{equation}

\noi
In the following, we first establish an estimate for fixed $t \in [-T, T]$.
Let $\s_1 > 0$.
By Sobolev's inequality,
$$
\big\|\jb{\nb}^{-\frac{1}{2}}(z_{1,N} w_{N}^2)(t) \big\|_{L^2}
\lesssim \big\|\jb{\nb}^{-\sigma_1}(z_{1,N} w_{N}^2)(t)\big\|_{L^r} ,
$$
provided that 
\begin{align}
\frac{1}{2}-\sigma_1 \geq \frac{3}{r}-\frac{3}{2}.
\label{w5}
\end{align}

\noi
By Lemma~\ref{LEM:gko}\,(ii), we have
$$
\big\|\jb{\nb}^{-\sigma_1}(z_{1,N} w_{N}^2)\big(t)\|_{L^r}
\les\|\jb{\nb}^{-\sigma_1}z_{1,N}(t)\|_{L^p}\|\jb{\nb}^{\sigma_1}(w_N^2)(t)\|_{L^q},
$$

\noi
provided that $0 \leq \s_1 \leq 1$ and 
\begin{equation}\label{zvezda}
\sigma_1 \geq   \frac{3}{p}+\frac{3}{q}-\frac{3}{r} .
\end{equation}

\noi
In the following, we will choose $p\gg 1$
such that $\s_1 > \frac{3}{q} - \frac 3r$ guarantees \eqref{zvezda}. 
By Lemma~\ref{LEM:gko}\,(i) 
and  Sobolev's inequality, 
we have 
\begin{align*}
\big\|\jb{\nb}^{\sigma_1}(w_N^2)(t)\big\|_{L^q}
& \lesssim\|\jb{\nb}^{\sigma_1}w_N(t)\|_{L^{q_1}}\|w_N(t)\|_{L^4}\\
& \lesssim \|\jb{\nb}^{\frac{1}{2}}w_N\|_{L^{2}}\|w_N(t)\|_{L^4}, 
\end{align*}

\noi 
provided that 
\begin{align}
\frac{1}{q}=\frac{1}{4}+\frac{1}{q_1}
\qquad \text{and}\qquad 
\frac{1}{2}-\sigma_1 \geq  \frac{3}{2}-\frac{3}{q_1}.
\label{w6}
\end{align}

\noi
In summary, if  the conditions
\eqref{w5}, \eqref{zvezda}, and \eqref{w6} are satisfied,
then we obtain the estimate 
\begin{equation}\label{zvezda2}
\big\|\jb{\nb}^{-\frac{1}{2}}(z_{1,N} w_{N}^2)(t)\big\|_{L^2}
\lesssim 
\|\jb{\nb}^{-\sigma_1}z_{1,N}(t)\|_{L^p}
\|\jb{\nb}^{\frac{1}{2}}(w_N)(t)\|_{L^{2}}
\|w_N(t)\|_{L^4}.
\end{equation}

\noi
Let us now show that we may ensure \eqref{w5}, \eqref{zvezda},  and \eqref{w6}. Since $p\gg 1$, it suffices to ensure that 
$$
\sigma_1>\frac{3}{q}-\frac{3}{r}=
\frac{3}{4}+\frac{3}{q_1}-\frac{3}{r}
\geq \frac{3}{4}+\frac{3}{2}-\bigg(\frac{1}{2}-\sigma_1\bigg)-\frac{3}{r}
\geq 2\sigma_1-\frac{1}{4}\,.
$$

\noi
This shows that  we can ensure \eqref{zvezda} and \eqref{w6} if $\sigma_1<\frac{1}{4}$. 
In this case, 
by  \eqref{zvezda2},  we arrive at the bound: 
\begin{align}
\big\|\jb{\nb}^{-\frac{1}{2}}(z_{1,N} w_N^2)\big\|_{L^1_T L^2_x}
\lesssim 
\|\jb{\nb}^{-\sigma_1}z_{1,N}\|_{L^{\frac{4}{3}}_T L^p_x}
\|\jb{\nb}^{\frac{1}{2}}w_N\|_{L^\infty_T L^{2}_x}
\|w_N\|_{L^4_T L^4_x}.
\label{w7}
\end{align}

\noi
Therefore, from \eqref{ik1} and \eqref{w7}
with the definition \eqref{X1} of the $X_T$-norm, 
we obtain
\begin{equation}
\big\|\L^{-1}(F_2(w_N))\big\|_{X_T}
\les T^\frac{1}{4}
\Big( \|z_{1,N}\|_{L^{2}_T W^{s_1, \infty}_x}
+ \|z_{2,N}\|_{L^{2}_T W^{s_1, \infty}_x}\Big)
\|w_N\|_{X_T}^2,
\label{F2}
\end{equation}

\noi
provided that 
\begin{align}
s_1 = -\s_1 > -\frac 14.
\label{s1}
\end{align}

\subsection{On the linear terms  in $w_N$}
\label{SUBSEC:pf3}

Let us next turn to the analysis of the terms linear in $w_N$.
By the Strichartz estimate  \eqref{strich}, we have 
\begin{align}
\label{ik2}
\begin{split}
\big\|\L^{-1}(F_1(w_N))\big\|_{X_T}
& \lesssim
\big\|\jb{\nb}^{-\frac{1}{2}}(Z_{2, N}w_{N})\big\|_{L^1_T L^2_x}
\\
& \hphantom{X}
+
\big\|\jb{\nb}^{-\frac{1}{2}}(z_{1,N} z_{2,N}  w_N)\big\|_{L^1_T L^2_x}
+\| (z_{2,N})^2 w_{N}\|_{L^{\frac{4}{3}}_{T, x}}.
\end{split}
\end{align}

\noi
We now evaluate each contribution on the right-hand side of \eqref{ik2}. 
By H\"older's inequality, we have 
\begin{align}
\| (z_{2,N})^2 w_{N}\|_{L^{\frac{4}{3}}_{, x}}
\leq T^\frac{1}{4} \|z_{2,N}\|_{L^8_T L^4_x}^2 \|w_N\|_{L^4_{T,x}}
\leq T^\frac{1}{4} \|z_{2,N}\|_{L^8_T W^{s_4, \infty}_x}^2  \|w_N\|_{X_T}, 
\label{w8}
\end{align}

\noi
provided that 
\begin{align}
 s_4 \geq 0.
\label{s2}
\end{align}

\noi
By Lemma \ref{LEM:gko} (ii), we have
\begin{align}
\begin{split}
\big\|\jb{\nb}^{-\frac{1}{2}}(Z_{2, N}w_{N})\big\|_{L^1_T L^2_x}
& \lesssim T^\frac{1}{2}
\|\jb{\nb}^{-\frac{1}{2}}Z_{2, N}\|_{L^2_T L^6_x}
\|\jb{\nb}^{\frac{1}{2}} w_N\|_{L^\infty_T L^2_x}\\
& \lesssim T^\frac{1}{2}
\|Z_{2, N}\|_{L^2_T W^{s_2, \infty}_x} \|w_N\|_{X_T}, 
\end{split}
\label{w9}
\end{align}

\noi
provided that 
\begin{align}
s_2 \geq -\frac 12.
\label{s3}
\end{align}

\noi
Finally, by applying Lemma \ref{LEM:gko} (ii) twice, 
we obtain
\begin{align}
\begin{split}
\big\|\jb{\nb}^{-\frac{1}{2}}(z_{1,N} z_{2,N}  w_N)\big\|_{L^1_T L^2_x}
& \lesssim 
\big\|\jb{\nb}^{s_1}(z_{1,N} z_{2,N})\big\|_{L^1_T L^6_x}
\|\jb{\nb}^{\frac{1}{2}} w_N\|_{L^\infty_T L^2_x}
\\
& \les T^\frac{1}{2}
\|z_{1,N} \|_{L^4_T W^{s_1, \infty}_x}
\|z_{2,N} \|_{L^4_T W^{s_4, \infty}_x}
\|w_N\|_{X_T}, 
\end{split}
\label{w10}
\end{align}

\noi
provided that 
\begin{align}
\max\Big(- \frac 12, -s_4\Big) \leq s_1 \leq 0.
\label{s4}
\end{align}

Therefore, putting \eqref{ik2}, \eqref{w8}, \eqref{w9}, and \eqref{w10}, 
we obtain 
\begin{align}
\begin{split}
\big\|\L^{-1}(F_1(w_N))\big\|_{X_T}
& \lesssim
T^\ta \Big\{
\|Z_{2, N}\|_{L^2_T W^{s_2, \infty}_x} 
+ \|z_{1,N} \|_{L^4_T W^{s_1, \infty}_x}
\|z_{2,N} \|_{L^4_T W^{s_4, \infty}_x}\\
& \hphantom{XXX}
+  \|z_{2,N}\|_{L^8_T W^{s_4, \infty}_x}^2
\Big\}
\|w_N\|_{X_T}
\end{split}
\label{F1}
\end{align}

\noi
for some $\ta > 0$ and $s_1, s_2,$ and $s_4$
satisfying \eqref{s2}, \eqref{s3}, and \eqref{s4}.

\subsection{On the source  terms} 
\label{SUBSEC:pf4}

We now estimate the contributions from the source terms.
Let $s_1$ and $s_4$ satisfy \eqref{s4}.
Then, 
by  Lemma \ref{LEM:gko} (ii) with H\"older's inequality
followed by Lemma~\ref{LEM:gko} (i), we have
\begin{align}
\begin{split}
\big\|\jb{\nb}^{-\frac{1}{2}}(z_{1,N}(z_{2,N})^2) \big\|_{L^1_T L^2_x}
& \leq
\big\|\jb{\nb}^{s_1}(z_{1,N}(z_{2,N})^2) \big\|_{L^1_T L^2_x}\\
& \les
\|\jb{\nb}^{s_1}z_{1,N} \|_{L^2_T L^4_x}
\Big\|\jb{\nb}^{s_4}\big((z_{2,N})^2\big) \Big\|_{L^2_T L^4_x}\\
& \les
T^\frac{1}{4} \|\jb{\nb}^{s_1}z_{1,N} \|_{L^4_{T, x}}
\|\jb{\nb}^{s_4}z_{2,N}\|_{L^4_T L^8_x}^2.
\end{split}
\label{w11}
\end{align}

\noi
Hence, 
from the Strichartz estimate \eqref{strich}
and \eqref{w11}, we obtain
\begin{align}
\begin{split}
\big\|\L^{-1}(F_0)\big\|_{X_T}
& \lesssim  
\|\jb{\nb}^{-\frac{1}{2}}Z_{5, N}\|_{L^1_T L^2_x}
+ \big\|\jb{\nb}^{-\frac{1}{2}}(z_{1,N}(z_{2,N})^2) \big\|_{L^1_T L^2_x}
+\|z_{2,N}\|_{L^4_T L^4_x}^3\\
& \lesssim  
T^\ta \Big\{\|Z_{5, N}\|_{L^2_T W^{s_5, \infty}_x}
+  \|z_{1,N} \|_{L^4_{T}W^{s_1, \infty}_x}
\|z_{2,N}\|_{L^4_T W^{s_4, \infty}_x}^2\\
& \hphantom{XXX}
+\|z_{2,N}\|_{L^8_T W^{s_4, \infty}_x}^3\Big\}
\end{split}
\label{F0}
\end{align}

\noi
for some $\ta > 0$, 
provided that $s_1$ and $s_4$ satisfy \eqref{s4}
and that $s_5$ satisfies
\begin{align}
s_5 \geq - \frac 12.
\label{s5}
\end{align}

\subsection{End of the proof}
\label{SUBSEC:end}

Let $s_1, s_2, s_4$, and $s_5$ satisfy \eqref{s1}, \eqref{s2}, \eqref{s3}, \eqref{s4}, and \eqref{s5}. 
Then, from \eqref{w4a}, \eqref{F3}, \eqref{F2}, \eqref{F1}, and \eqref{F0}, 
we have
\begin{equation*}
\begin{split}
 \|w_N\|_{X_T}
& \leq C
T^\ta A^{(1)}_N + CT^\ta 
A^{(2)}_N
\|w_N\|_{X_T}\\
& \hphantom{X} + C T^\ta
\bigg(\sum_{j = 1}^2 
\|z_{j,N}\|_{L^{2}_T W^{s_j, \infty}_x}\bigg)
\|w_N\|_{X_T}^2
+  C\|w_N\|_{X_T}^3, 
\end{split}
\end{equation*}

\noi
where $A^{(1)}_N $ and $A^{(2)}_N $ are defined by 
\begin{align}
\begin{split}
A^{(1)}_N & = \|Z_{5, N}\|_{L^2_T W^{s_5, \infty}_x}
+  \|z_{1,N} \|_{L^4_{T}W^{s_1, \infty}_x}
\|z_{2,N}\|_{L^4_T W^{s_4, \infty}_x}^2 
+\|z_{2,N}\|_{L^8_T W^{s_4, \infty}_x}^3, \\
A^{(2)}_N & = 
\|Z_{2, N}\|_{L^2_T W^{s_2, \infty}_x} 
+ \|z_{1,N} \|_{L^4_T W^{s_1, \infty}_x}
\|z_{2,N} \|_{L^4_T W^{s_4, \infty}_x}
+  \|z_{2,N}\|_{L^8_T W^{s_4, \infty}_x}^2.
\end{split}
\label{W2}
\end{align}

\noi	
Suppose that 
\begin{align}
\begin{split}
 R(T) := 
\sup_{N \in \N}  \max\Big(
& \|z_{1,N} \|_{L^4_T W_x^{s_1, \infty}}, 
 \|z_{2,N} \|_{L^8_T  W_x^{s_4, \infty}}, \\
&   \|Z_{2, N}\|_{L^2_T  W_x^{s_2, \infty}},  
\|Z_{5, N}\|_{L^2_T  W_x^{s_5, \infty}}\Big) 
\leq T^{\ta_0}
\end{split}
\label{W3}
\end{align}

\noi
for some $\ta_0 > 0$.
Then, it follows from a standard continuity argument
that there exists $T_0 > 0$ such that 
\begin{align*}
\| w_N \|_{X_T} \leq C(R) T^\ta
\end{align*}

\noi	
for any $0 < T \leq  T_0$, 
uniformly in $N \in \N$.
Here, we used the fact that $(w, \dt w) |_{t = 0} = (0, 0)$.

Let $M \geq N \geq 1$.
Note that $F_j$, $j = 0, 1, 2, 3$, are multilinear in $w_N$
and the stochastic terms $z_{1, N}$, $z_{2, N}$, $Z_{2, N}$, and $Z_{5, N}$.
Then, by proceeding as in Subsections \ref{SUBSEC:pf2}, \ref{SUBSEC:pf3}, and \ref{SUBSEC:pf4}, 
 we also obtain the following difference estimate:
\begin{equation}
\begin{split}
 \|w_M - w_N\|_{X_T}
& \leq C
T^\ta B^{(1)}_{M, N}
+ CT^\ta B^{(2)}_{M, N} \|w_N\|_{X_T}
+ CT^\ta A^{(2)}_N\|w_M  - w_N\|_{X_T}\\
& \hphantom{X} 
+ C T^\ta
\bigg(\sum_{j = 1}^2 
\|z_{j,M} - z_{j,N}\|_{L^{2}_T W^{s_j, \infty}_x}\bigg)
\|w_M\|_{X_T}^2\\
& \hphantom{X} 
+ C T^\ta
\bigg(\sum_{j = 1}^2 
\|z_{j,N}\|_{L^{2}_T W^{s_j, \infty}_x}\bigg)
\big(\|w_M \|_{X_T}+ \| w_N\|_{X_T}\big)
\|w_M - w_N\|_{X_T}\\
& \hphantom{X} 
+ C\big(\|w_M \|_{X_T}^2+ \| w_N\|_{X_T}^2\big)
\|w_M - w_N\|_{X_T}
\end{split}
\label{W5}
\end{equation}

\noi
where
$B^{(1)}_{M, N}$ and $B^{(2)}_{M, N}$ are defined by 
\begin{align*}
B^{(1)}_{M, N}
& = 
\|Z_{5, M} - Z_{5, N}\|_{L^2_T W^{s_5, \infty}_x}
+  \|z_{1, M}  - z_{1,N} \|_{L^4_{T}W^{s_1, \infty}_x}
 \|z_{2,M}\|_{L^4_T W^{s_4, \infty}_x}^2 \\
& \hphantom{X}
+  \|z_{1,N} \|_{L^4_{T}W^{s_1, \infty}_x}
\|z_{2, M} - z_{2,N}\|_{L^4_T W^{s_4, \infty}_x}
\big(\|z_{2,M}\|_{L^4_T W^{s_4, \infty}_x} + \|z_{2,N}\|_{L^4_T W^{s_4, \infty}_x}\big) \\
& \hphantom{X}
+\big(
 \|z_{2,M}\|_{L^8_T W^{s_4, \infty}_x}^2 + 
 \|z_{2,N}\|_{L^8_T W^{s_4, \infty}_x}^2\big)\|z_{2, M} - z_{2,N}\|_{L^8_T W^{s_4, \infty}_x}, \\
B^{(2)}_{M, N}
& =  
\|Z_{2, M} - Z_{2, N}\|_{L^2_T W^{s_2, \infty}_x} 
+ \|z_{1, M} - z_{1,N} \|_{L^4_T W^{s_1, \infty}_x}
\|z_{2,M} \|_{L^4_T W^{s_4, \infty}_x}\\
& \hphantom{X}
 + \|z_{1,N} \|_{L^4_T W^{s_1, \infty}_x}
\|z_{2, M} - z_{2,N} \|_{L^4_T W^{s_4, \infty}_x}\\
& \hphantom{X}
+ \big( \|z_{2,M}\|_{L^8_T W^{s_4, \infty}_x} +  \|z_{2,N}\|_{L^8_T W^{s_4, \infty}_x} \big) 
 \|z_{2, M} - z_{2,N}\|_{L^8_T W^{s_4, \infty}_x}.
\end{align*}

In addition to the assumption \eqref{W3}, we now suppose that 
as $N \to \infty$, 
$z_{1, N}$, $z_{2, N}$, $Z_{2, N}$, and $Z_{5, N}$
converge to 
the limits
$z_{1}$, $z_{2}$, $Z_{2}$, and $Z_{4}$	
in $C([-T, T]; W^{s, \infty}(\T^3))$
for $s = s_1, s_4, s_2$, and $s_5$, respectively.
Then, from \eqref{W5}, we obtain 
\begin{equation*}
 \|w_M - w_N\|_{X_T}
 \leq  C(R)T^\ta \|w_M  - w_N\|_{X_T} + o_{M, N\to \infty}(1)
\end{equation*}
\noi
Then, by possibly making $T_0 > 0$ smaller, 
we conclude that 
\[ \|w_N-w_M\|_{X_T} \too 0\]

\noi
for any $0 < T \leq T_0$
as $M, N \to \infty$.
This  implies that $w_N$ converges to some $w$
in $X_T$ as $N \to \infty$.
Recalling the decomposition 
 $u_{N}=z_{1,N}+z_{2,N}+w_{N}$, we 
conclude that $u_N$ converges
to $u = z_1 + z_2 + w$
in $C([-T, T]; H^{s_1}(\T^3))$
as $N \to \infty$.

It remains to check that 
  the assumption \eqref{W3}
and the assumption on the convergence of 
$z_{1, N}$, $z_{2, N}$, $Z_{2, N}$, and $Z_{5, N}$
hold true with large probability.
By choosing $s_1 = \al - \frac 32 - \eps$, 
$s_2 = 2(\al - \frac 32) -\eps$, 
$s_4 = 3(\al - \frac 32)+1 -\eps$, 
and $s_5 =  2(\al - \frac 32) -\eps$
for some small $\eps > 0$, 
it is easy to see that 
the conditions
 \eqref{s1}, \eqref{s2}, \eqref{s3}, \eqref{s4}, and \eqref{s5}
 are satisfied for $\frac 54 < \al \leq \frac 32$. 
 (Note that the restriction $\al > \frac 54$ appears in 
  \eqref{s1}, \eqref{s3}, \eqref{s4}, and~\eqref{s5}.)
Therefore, 
it follows from 
Proposition \ref{PROP:ran1}, \ref{PROP:z2}, and~\ref{PROP:Z4}
that 
there exists
a set  $\O_T$ of complemental probability smaller than $C\exp(-c/T^\kk)$ 
such that 
   the assumption~\eqref{W3}
and the assumption on the convergence of 
$z_{1, N}$, $z_{2, N}$, $Z_{2, N}$, and $Z_{5, N}$
hold true on $\O_T$, allowing us to prove
the convergence of $u_N$ to $u$ 
in $C([-T, T]; H^{s_1}(\T^3))$ as above.
This completes the proof of  Theorem~\ref{THM:3}.

\section{On the triviality of the limiting dynamics without renormalization}
\label{SEC:tri}

\subsection{Reformulation of the problem}

Fix $1 \leq  \al \leq \frac 32$
and a pair $(w_0, w_1) \in \H^\frac{3}{4}(\T^3)$.
Let $u_N$ be the solution to the (un-renormalized)
NLW~\eqref{KG}
with the following initial data:
\begin{align}
 (u_N, \dt u_N)|_{t= 0} 
= (w_0, w_1) + 
(\wt u_{0, N}^\o, \wt u_{1, N}^\o), 
\label{Y0}
\end{align}

\noi
where the random initial data 
$(\wt u_{0, N}^\o, \wt u_{1, N}^\o)$ is given by 
\eqref{series_N2}
with $C_N >0$ implicitly defined as in \eqref{CN2}.
In this section, we present the proof of Theorem \ref{THM:Tri}
by reformulating the Cauchy problem for $u_N$ 
as
\begin{equation}\label{KG2}
\begin{cases}
\L_N  u_N+u_N^3 - C_N u_N=0\\
(u_N, \dt u_N)|_{t = 0} = (w_0, w_1) + (\wt u^\o_{0, N}, \wt u^\o_{1, N}), 
\end{cases}
\end{equation}

\noi
where 
 $\L_N=\partial_t^2 -\Delta +C_N$ as in \eqref{lin2}.

Since $C_N$ in \eqref{CN2} is implicitly defined, 
we first need to study  the asymptotic behavior of $C_N$ as $N \to \infty$.

\begin{lemma}\label{LEM:CN}
Let $1\leq  \al \leq \frac 32$.
Then,  for each $N \in \N$, 
there exists a unique number $C_N \geq 1$ satisfying the equation \eqref{CN2}.
Moreover, we have 
\begin{align}
C_N = 3\s_N + R_N
\label{CN4}
\end{align}

\noi
for all sufficiently large $N \gg1 $, 
where $\s_N  = \sum_{|n|\leq N} \jb{n}^{-2\al}$ is as in \eqref{CN}
and the error term $R_N$ satisfies
\begin{align*}
|R_N| \sim 
\begin{cases}
\log \log N, & \text{for } \al = \frac 32, \rule[-3mm]{0pt}{0pt}\\
N^{\frac{1}{2}(3-2\al)^2}, & \text{for }  1 \leq \al  < \frac 32.
\end{cases}
\end{align*}

\noi
In particular, we have $R_N = o(\s_N)$ as $N \to \infty$.

\end{lemma}

\begin{proof}
Let $C_N$ be as in \eqref{CN2}.
As $C_N$ increases  from $0$ to $\infty$, 
the right-hand side of \eqref{CN2} decreases from $\infty$  to $0$.
Hence, for each $N \in \N$, there exists a unique solution $C_N > 0$ to~\eqref{CN2}. 

Suppose that $C_N < 1$ for some $N \in \N$. Then, considering the contribution
from $n = 0$ on the right-hand side of \eqref{CN2}, 
we obtain $C_N \geq 3$, leading to a contradiction.
Hence, we must have $C_N \geq1$ for any $N \in \N$.

We first consider the case $1\leq \al < \frac 32$.
Since $C_N \geq 1$, it follows from \eqref{CN2}
that 
$C_N \les N^{3 - 2\al}$.
Using this upper bound on $C_N$, 
we estimate the contribution from $|n| \sim N$:
\begin{align*}
 C_N  \ges 
  \sum_{|n|\leq N }\frac{1} {(N^{3-2\al} + |n|^2) \jb{n}^{2(\al-1)}}
  \ges 
  \sum_{|n|\sim N }\frac{1} { \jb{n}^{2\al}} \sim N^{3-2\al}, 
\end{align*}

\noi
where we used the assumption $\al \geq  1$ in the second step.
This shows that $C_N \sim N^{3 - 2\al}$.
Using this asymptotic behavior with \eqref{CN}, 
we then obtain
\eqref{CN4}
with the error term  $R_N$  given by 
\begin{align}
R_N 
& =  3 \sum_{|n|\leq N }\frac{1} { \jb{n}^{2(\al-1)}}
\bigg( \frac{1}{(C_N + |n|^2)} - \frac{1}{\jb{n}^2}\bigg).
\label{CN5}
\end{align}

\noi
By separately estimating the contributions
from $\big\{|n| \ll N^{\frac 32 - \al}\big\}$
and $\big\{  N^{\frac 32 - \al} \leq |n| \leq N\big\}$, 
we have
\begin{align*}
|R_N| 
 & = 3 \sum_{|n|\leq N }\frac{1} { \jb{n}^{2\al}}
 \frac{C_N - 1}{(C_N + |n|^2)}\\
& \sim N^{(\frac 32 - \al)(3-2\al)}.
\end{align*}

Next, we consider the case $\al = \frac 32$.
Proceeding as above, 
we immediately see that $C_N \sim \log N$.
The contribution to $R_N$ in \eqref{CN5}
from $\big\{| n |\ges \sqrt{\log N}\big\}$
is $O(1)$, 
while the  contribution to $R_N$ in \eqref{CN4}
from $\big\{| n| \ll  \sqrt{\log N}\big\}$
is $O(\log \log N)$.
This completes the proof of  Lemma~\ref{LEM:CN}.
\end{proof}

\subsection{On the Strichartz estimates with a parameter}

In order to study the equation~\eqref{KG2}, 
we  review the relevant Strichartz estimates
for the Klein-Gordon operator with a general mass.
Given $a\geq 1$, 
with a slight abuse of notation, define $\L_a$ by 
\begin{align*}
\L_a :=\partial_t^2 -\Delta +a.
\end{align*}

\noi
Let  $\L_a^{-1}$ be  the Duhamel integral operator given by 
\begin{align*}
\L_a^{-1} F(t) = \int_0^t \frac{\sin((t-t')\sqrt{a - \Dl} )}{\sqrt{a - \Dl}} F(t') dt'.
\end{align*}

\noi
Namely, 
 $u :=\L_a^{-1}(F)$ is the solution to 
 the following nonhomogeneous linear equation:
\begin{equation*}
\begin{cases}
\L_a u= F\\ 
(u,\partial_t u)|_{t=0}=(0,0).
\end{cases}
\end{equation*}

\noi
Then, by making systematic modifications of the proof of Lemma \ref{LEM:Str}
on $\R^3$ (see \cite{tz-cime}) and applying the finite speed of propagation,
we see that the same non-homogeneous Strichartz estimate as \eqref{strich} holds, 
uniformly in $a \geq 1$:
\begin{align}
\|\L_a^{-1}(F)\|_{X_T}
\lesssim 
\min\Big(
\|F\|_{L^1([-T,T];H^{-\frac{1}{2}}(\T^3))}
,
\|F\|_{L^{\frac{4}{3}}([-T,T]\times\T^3)}
\Big)
\label{Stri2}
\end{align}

\noi
for any  $0 < T \leq 1$, 
where the $X_T$-norm is defined in \eqref{X1}.

We also record the following lemma on the linear solution
associated with $\L_a$, $a\geq 1$.

\begin{lemma}\label{LEM:lina}
Given $a\geq 1$, 
define
$S_a(t)$ by 
\[ S_a(t) (w_0, w_1) = \cos (t\sqrt{a - \Dl}) w_0 + \frac{\sin(t \sqrt{a-\Dl})}{\sqrt{a-\Dl}}w_1.\]

\noi
Then, there exists $C>0$ such that 
\begin{align}
 \|S_a(t) (w_0, w_1)\|_{X_T}
 \le C \|(w_0, w_1) \|_{\H^\frac{3}{4}}
 \label{lina1}
\end{align}

\noi
for
any  $(w_0, w_1) \in \H^\frac{3}{4}(\T^3)$ and $0 <  T \leq 1$, uniformly in  $a \geq 1$.
Moreover, $S_a(t) (w_0, w_1)$ tends to $0$ 
  in the space-time distributional sense as $a \to \infty$. 
\end{lemma}

\begin{proof}

The estimate \eqref{lina1} follows
easily from H\"older's inequality in $t$
and Sobolev's inequality in $x$
along with the boundedness of $S_a(t)$ in $\H^\frac{3}{4}(\T^3)$.
As for the second claim, 
we only consider $e^{it\sqrt{a - \Dl}} f$
for $f \in L^2(\T^3)$.
Note that, for each fixed $n \in \Z^3$, 
$\sqrt{a + |n|^2} - \sqrt a$ tends to 0 as $a \to \infty$.
Then, by 
the dominated convergence theorem (for the summation in $n \in \Z^3$)
and 
the Riemann-Lebesgue lemma (for the integration in $t$), we have
\begin{align*}
\begin{split}
\lim_{a \to \infty}\iint 
& \big(e^{it\sqrt{a - \Dl}} f\big)(x) \cj{\phi(t, x)} \, dx dt\\
& = \lim_{a \to \infty}\int e^{it\sqrt a} \bigg( \sum_{ n \in \Z^3} 
e^{it (\sqrt{a + |n|^2} - \sqrt a)}
\ft f(n) \cj{\ft \phi (t, n)} \bigg) dt \\
& = \lim_{a \to \infty}\int e^{it\sqrt a} \jb{f, \phi(t)}_{L^2_x} dt \\
& = 0
\end{split}
\end{align*}

\noi
for any test function $\phi \in C^\infty(\R\times \T^3)$ with a compact support in $t$.
\end{proof}

\begin{remark}\label{REM:Str}\rm

\noi
Let $a\geq 1$.
Then, we have the following homogeneous Strichartz estimate:
\begin{align}
 \| S_a(t) (w_0, w_1)\|_{L^q([0, 1];  L^r(\T^3))} \leq C \| (w_0, w_1) \|_{H^\frac{2}{q}_a\times H^{\frac{2}{q}-1}_a(\T^3)}
\label{lin5}
\end{align}

\noi
for $2 < q \leq \infty$ and $\frac1q + \frac 1r = \frac 12$, 
where the $H^s_a$-norm is defined by 
\[ \| f\|_{H^s_a} = \bigg(\sum_{n \in \Z^3} (a + |n|^2)^s |\ft f(n)|^2 \bigg)^\frac{1}{2}.\]

\noi
The proof of \eqref{lin5} follows from a straightforward modification
of the standard homogeneous Strichartz estimate (i.e.~$a = 1$).
For $s >0$, 
the $H^s_a$-norm diverges as $a \to \infty$
and hence the homogeneous Strichartz estimate~\eqref{lin5}
is not  useful for our application.

\end{remark}

\subsection{Proof of Theorem \ref{THM:Tri}}

Let $\wz_{1, N}$ and $\ws_N$ be as in \eqref{T1}
and \eqref{T3}.
As in \eqref{G3}, \eqref{z2}, and \eqref{B1}, we define 
\begin{align}
\begin{split}
\wZ_{1,N}:= \wz_{1,N},
\quad 
 \wZ_{2,N}  & :=(\wz_{1,N})^2-\ws_{N},
\quad  \wZ_{3,N}:=(\wz_{1,N})^3- 3\ws_{N} \wz_{1,N}, \\
\wZ_{4, N} & : =  \wz_{2,N}: =-\L_N^{-1}\big((\wz_{1,N})^3-3\ws_N \wz_{1,N}\big), \\
 \wZ_{5,N}  & :  =\big\{(\wz_{1,N})^2-\ws_{N}\big\} \wz_{2,N}, 
\end{split}
\label{Y1}
\end{align}

\noi
where $\L_N$ is as  in \eqref{lin2}.
Then, by repeating the arguments in Sections \ref{SEC:sto1} and \ref{SEC:sto2}, 
we see that the analogues of Propositions
\ref{PROP:ran1},   \ref{PROP:z2}, and \ref{PROP:Z4}
hold for 
$\wZ_{j,N}$, $j = 1, \dots, 5$.
In the following lemma, we summarize
the regularity and convergence properties
of these stochastic terms.

\begin{lemma}\label{LEM:stoa}

Let $1< \al <  \frac 32$
and $s_j$, $j = 1, \dots, 5$, 
satisfy the regularity assumptions \eqref{reg1}, \eqref{reg_z2}, and \eqref{reg2}.
Fix $j = 1, \dots, 5$.
 Then, given any $T>0$, $\wZ_{j, N}$ converges
almost surely 
to $0$ in $C([-T, T]; W^{s_j, \infty}(\T^3))$
as $N \to \infty$.
Moreover, 
given $2\leq q < \infty$, 
there exist 
positive constants $C$, $c$, $\kappa$, $\ta$ 
and small $\dl > 0$
such that for every $T> 0$, 
 there exists a  set $\O_T$ of complemental probability smaller than $C\exp(-c/T^\kk)$ 
such that 
\begin{equation}\label{conv1a}
\big\|\wZ_{j,N}\big\|_{L^q([-T,T];W^{s_j, \infty}(\T^3))}\leq C_N^{-\dl} T^\ta
\end{equation}

\noi
for any  $\o\in\O_T$
and any $ N\geq 1$.
In particular, 
for any $\o \in \O_T$, 
 $\wt Z_{j, N}$ tends to $0$ in $L^q([-T,T];W^{s_j, \infty}(\T^3))$
 as $N \to \infty$.

When $\al = \frac 32$, 
the same result holds but only along a subsequence $\{N_k \}_{k \in \N}$.

\end{lemma}

\begin{proof}
We only consider the case $j = 1$ since the other cases follow in a similar manner.
Fix $N \in \N$.  With  $\nbn$ as in \eqref{N1}, 
let
\begin{align*}
g^{t, N}_n(\o):=\cos(t\nbn)\, g_{n}(\o)+ \sin(t\nbn)\,h_n(\o).
\end{align*}

\noi
Then, from \eqref{T2}, we have
$$
\jb{\nb}^{s_1}\wZ_{1,N}
=
\sum_{|n|\leq N } \frac{g^{t, N}_n(\o)}{\nbn \jb{n}^{\al-1- s_1}}e^{in\cdot x}.
$$

\noi
Let $q, r< \infty$.  
Then, proceeding as in \eqref{G4a}
with $\nbn \geq \max\big(C_N^\frac{1}{2}, \jb{n}\big)$,  we 
have
\begin{align*}
\begin{split}
\Big\|\big\|\jb{\nb}^{s_1}
 Z_{1,N}\big\|_{L^q_TL^r_x}\Big\|_{L^p(\O)}
&   \leq
\Big\|\big\| \jb{\nb}^{s_1}Z_{1,N}(t, x)
\big\|_{L^p(\O)}
\Big\|_{L^q_TL^r_x}\\
& %
\les T^\frac{1}{q}p^\frac{1}{2} \bigg(\sum_{|n|\leq N } \frac{1}{\nbn^2\jb{n}^{2(\al-1-s_1)}}\bigg)^\frac{1}{2}
\les  C_N^{ - \dl_0}T^{\frac{1}{q}} p^\frac{1}{2}
\end{split}
\end{align*}

\noi
for any $p \geq \max(q, r)$ 
and sufficiently small $\dl_0 > 0$ such that $2(\al - s_1 -2 \dl_0) > 3$.
By Chebyshev's inequality, 
we then have
\begin{align}
P\Big( \big\|\wZ_{j,N}\big\|_{L^q([-T,T];W^{s_j, \infty}(\T^3))}>  C_N^{-\dl} T^\ta\Big)
\leq C \exp\bigg(- c \frac{C_N^{2(\dl_0 - \dl)}}{T^{2(\frac{1}{q} - \ta)}}\bigg).
\label{N2}
\end{align}

\noi
In view of Lemma \ref{LEM:CN} with \eqref{CN}, 
by choosing $\dl, \ta > 0$ sufficiently small, 
the right-hand side of \eqref{N2} is summable
over $N \in \N$, 
as long as  $1 < \al < \frac 32$.
Define
$\O_T$ by 
\begin{align}
 \O_T = \bigcap_{N \in \N}
\Big\{ \o \in \O: \big\|\wZ_{j,N}\big\|_{L^q([-T,T];W^{s_j, \infty}(\T^3))}\leq  C_N^{-\dl} T^\ta\Big\}.
\label{N3}
\end{align}

\noi
Then, we have $P(\O_T^c) \leq
C\exp(-c/T^\kk)$ 
and  \eqref{conv1a} holds
for any $\o \in \O_T$ and any $N \in \N$, 
when $1 < \al < \frac 32$.
When $ \al = \frac 32$, 
we need to choose a subsequence~$\{N_k \}_{k \in \N}$
growing sufficiently fast such that the right-hand side
of \eqref{N2} is summable along this subsequence~$\{N_k \}_{k \in \N}$.
Then, we define $\O_T$ as in \eqref{N3}
but by taking an intersection over $\{N_k\}_{k \in \N}$.
This 
 yields \eqref{conv1a} 
 for any $\o \in \O_T$ and any $N = N_k$, $k \in \N$,  when $\al = \frac 32$.

The rest follows exactly as in the proofs of 
Lemma \ref{LEM:nel_pak}
and Propositions \ref{PROP:ran1}, \ref{PROP:z2}, and~\ref{PROP:Z4}.
Lastly, 
in view of 
 Fatou's lemma
and  the asymptotic behavior $C_N \to \infty$, 
we conclude  from~\eqref{conv1a} that $ \wZ_{1, N}$ tends to $0$
(along the subsequence $\{N_k\}_{k \in \N}$ when $\al = \frac 32$).
\end{proof}

With  Lemma \ref{LEM:stoa} in hand, 
we can proceed as in Section \ref{SEC:LWP}.\footnote{In the following, 
it is understood that when $\al = \frac 32$, 
we work on the subsequence $\{N_k\}_{k \in \N}$
from Lemma~\ref{LEM:stoa} instead of the whole natural numbers $N \in \N$.}
Namely, given  $(w_0, w_1) \in \H^\frac{3}{4}(\T^3)$, 
let $u_N$ be the solution to the (un-renormalized)
NLW~\eqref{KG}
with the  initial data in~\eqref{Y0}:
\[ (u_N, \dt u_N)|_{t= 0} 
= (w_0, w_1) + 
(\wt u_{0, N}^\o, \wt u_{1, N}^\o),\]

\noi
where
$(\wt u_{0, N}^\o, \wt u_{1, N}^\o)$ is 
the truncated random initial data defined in~\eqref{series_N2}.
Now, we write
\begin{align}
\wt u_{N}=\wz_{1,N}+\wz_{2,N}+\wt w_{N},
\label{Y3}
\end{align}

\noi
where $\wz_{1, N}$ and $\wz_{2, N}$ are  as in \eqref{T1} and \eqref{Y1}, respectively.
Recalling that $u_N$ also satisfies~\eqref{KG2}, 
we see that $\wt w_N$
is the solution to 
\begin{equation}
\begin{cases}
\L_N \wt w_N+\wt F_0+\wt F_1(\wt w_N)+\wt F_2(\wt w_N)+\wt F_3(\wt w_N)=0\\
(\wt w_N, \dt  \wt w_N)|_{t=0}=(w_0, w_1),
\end{cases}
\label{Y4}
\end{equation}

\noi
where $\L_N$ is as in \eqref{lin2} and $\wt F_j$, $j = 0, \dots, 3$, are given by 
\begin{align*}
\wt F_0& =3\wZ_{5, N}+ 3\wz_{1,N}(\wz_{2,N})^2+(\wz_{2,N})^3,\\
\wt F_{1}(\wt w_N)& =3\wZ_{2, N}\wt w_N+6\wz_{1,N}  \wz_{2,N} \wt w_N+3 (\wz_{2,N})^2 \wt w_{N},\\
\wt F_2(\wt w_N)& =3\wz_{1,N} (\wt w_N)^2+3 \wz_{2,N}\wt w_N^2,\\
\wt  F_3(\wt w_N) & =\wt  w_N^3.
\end{align*}

Given $N \in \N$, 
define $S_N(t)$ by 
\[ S_N(t) (w_0, w_1) = \cos (t\sqrt{C_N - \Dl}) w_0 + \frac{\sin(t \sqrt{C_N-\Dl})}{\sqrt{C_N-\Dl}}w_1.\]

\noi
Then, the Duhamel formulation of \eqref{Y4} 
is given by 
\begin{align*}
\wt  w_{N}=S_N(t)(w_0, w_1) + 
\L_N^{-1}(\wt F_0 + \wt F_1(\wt w_N)
+ \wt F_2(\wt w_N) + \wt F_3(\wt w_N)), 
\end{align*}

\noi
Define 
$\wt A^{(1)}_N $,  $\wt A^{(2)}_N $, and $\wt R(T)$
by replacing $z_{j, N}$ and $Z_{j, N}$ in \eqref{W2} and \eqref{W3} 
with $\wz_{j, N}$ and $\wZ_{j, N}$.
Then, 
by repeating the analysis in Section \ref{SEC:LWP}
with \eqref{Stri2}, 
we obtain
\begin{equation}
\begin{split}
 \|\wt w_N\|_{L^\infty_TH^\frac{1}{2}_x}
& \leq \| (w_0, w_1)\|_{\H^\frac{1}{2}}
+ C
T^\ta \wt A^{(1)}_N + CT^\ta 
\wt A^{(2)}_N
\|\wt w_N\|_{X_T}\\
& \hphantom{X} + C T^\ta
\bigg(\sum_{j = 1}^2 
\|\wz_{j,N}\|_{L^{2}_T W^{s_j, \infty}_x}\bigg)
\|\wt w_N\|_{X_T}^2
+  C\|\wt w_N\|_{L^4_{T, x}}^3
\end{split}
\label{Y6}
\end{equation}

\noi
and 
\begin{equation}
\begin{split}
 \|\wt w_N\|_{L^4_{T, x}}
& \leq \| S_N(t)(w_0, w_1)\|_{L^4_{T, x}}
+ C
T^\ta \wt A^{(1)}_N + CT^\ta 
\wt A^{(2)}_N
\|\wt w_N\|_{X_T}\\
& \hphantom{X} + C T^\ta
\bigg(\sum_{j = 1}^2 
\|\wt z_{j,N}\|_{L^{2}_T W^{s_j, \infty}_x}\bigg)
\|\wt w_N\|_{X_T}^2
+  C\|\wt w_N\|_{L^4_{T, x}}^3, 
\end{split}
\label{Y7}
\end{equation}

\noi
where the constants are independent of $N \in \N$, thanks
to the uniform Strichartz estimate~\eqref{Stri2}.
By H\"older's inequality in $t$ and Sobolev's inequality in $x$
(as in the proof of Lemma~\ref{LEM:lina}), 
we have
\begin{align}
\| S_N(t)(w_0, w_1)\|_{L^4_{T, x}}
\les T^\frac{1}{4} \| (w_0, w_1) \|_{\H^\frac{3}{4}}, 
\label{Y8}
\end{align}

\noi
uniformly in $N \in \N$.
Then, 
it follows from \eqref{Y6}, \eqref{Y7},  and \eqref{Y8}
that there exists small $T_1 > 0$ depending on $\wt R(T)$
such that 
\begin{align}
\begin{split}
 \|\wt w_N\|_{L^\infty_TH^\frac{1}{2}_x} 
 & \leq 2\| (w_0, w_1)\|_{H^\frac{1}{2}_x},\\
\| \wt w_N \|_{L^4_{T, x}} & \leq \big(1 + C(\wt R(T))\big) T^\ta
\end{split}
\label{Y9}
\end{align}

\noi
for any $0 < T \leq  T_1$, 
uniformly in $N \in \N$.
It follows from   Lemma \ref{LEM:stoa}
that 
for each small $0 < T \leq T_1$, 
there exists 
a set  $\O_T$ of complemental probability smaller than $C\exp(-c/T^\kk)$ 
such that 
\begin{align}
\wt R(T) \leq C_N^{-\dl} T^\ta.
\label{Y9a}
\end{align}

In the following, we fix $\o \in \O_T$
and  show that $\wt w_N$ tends to $0$ as a space-time distribution as $N \to \infty$.
From Lemma \ref{LEM:stoa} with \eqref{Y9} and \eqref{Y9a}, 
we see that 
\begin{align}
\L_N^{-1}(\wt F_0 + \wt F_1(\wt w_N)
+ \wt F_2(\wt w_N))
\too 0
\label{Y10}
\end{align}

\noi
in $X_T$ as $N \to \infty$.
On the other hand, 
by Sobolev's inequality (with $\dl > 0$ sufficiently small) and Lemma \ref{LEM:CN}, 
we have
\begin{align}
\big\|\L_N^{-1} ( \wt F_3(\wt w_N)\big\|_{L^\infty_T L^2_x}
\le C_N^{-\dl} \| \wt w_N^3\|_{L^1_T H^{-1+2\dl}_x}
\les C_N^{-\dl} \| \wt w_N\|_{L^4_{T, x}}^3
\too 0
\label{Y11}
\end{align}

\noi
as $N \to \infty$.
Therefore from Lemmas \ref{LEM:CN} and \ref{LEM:lina}
with \eqref{Y10} and \eqref{Y11}, 
we conclude that $\wt w_N$
tends to $0$ in the space-time distributional sense.

Finally, from the decomposition \eqref{Y3}, 
 Lemma \ref{LEM:stoa}, and the convergence property of $\wt w_N$
 discussed above, 
 we conclude that, for each $\o \in \O_T$, 
  $\wt u_N$ converges to $0$
 as space-time distributions on $[-T, T]\times \T^3$ 
  as $N \to \infty$.
This completes the proof of Theorem \ref{THM:Tri}.

\begin{ackno}\rm 

T.O.~was supported by the European Research Council (grant no.~637995 ``ProbDynDispEq''
and grant no.~864138 ``SingStochDispDyn").
N.T.~was supported by the ANR grant ODA (ANR-18-CE40-0020-01).
O.P.~was supported by 
 the EPSRC New Investigator Award 
 (grant no.~EP/S033157/1).
O.P.~would like to express her sincere gratitude to Professor Ioan Vrabie for his support and for his teaching, especially for 
getting her interested in mathematical analysis. 
The authors would like to thank the anonymous referee for helpful comments.

\end{ackno}


\begin{thebibliography}{99}


\bibitem{russo4}
S.~Albeverio, Z.~Haba, F.~Russo, 
{\it Trivial solutions for a nonlinear two-space-dimensional wave equation perturbed by space-time white noise,} Stochastics Stochastics Rep.  56 (1996), no.~1-2, 127--160. 


\bibitem{russo5}
S.~Albeverio, Z.~Haba, F.~Russo, 
{\it A two-space dimensional semilinear heat equation perturbed by (Gaussian) white noise,}
Probab. Theory Related Fields 121 (2001), no. 3, 319--366.

\bibitem{AK} S.~Albeverio, S.~Kusuoka {\it The invariant measure and the flow associated to the $\Phi^4_3$-quantum field model},
to appear in Ann. Sc. Norm. Super. Pisa Cl. Sci. 


\bibitem{BCD}
H.~Bahouri, J.-Y.~Chemin, 
R.~Danchin, 
{\it Fourier analysis and nonlinear partial differential equations,}
Grundlehren der Mathematischen Wissenschaften [Fundamental Principles of Mathematical Sciences], 
343. Springer, Heidelberg, 2011. xvi+523 pp. 


%
\bibitem{Bass}
R.~Bass, 
{\it Stochastic processes.}
Cambridge Series in Statistical and Probabilistic Mathematics, 33. Cambridge University Press, Cambridge, 2011. xvi+390 pp. 




\bibitem{Bring2}
B.~Bringmann, 
{\it Invariant Gibbs measures for the three-dimensional wave equation with a Hartree nonlinearity II: Dynamics}, 
arXiv:2009.04616 [math.AP].

\bibitem{BOP1} A.~B\'enyi, T.~Oh, O.~Pocovnicu,
{\it Wiener randomization on unbounded domains and an application to almost sure well-posedness of NLS}, Excursions in harmonic analysis. Vol. 4, 3--25, Appl. Numer. Harmon. Anal., Birkh\"auser/Springer, Cham, 2015. 

\bibitem{BOP} A.~B\'enyi, T.~Oh, O.~Pocovnicu,
 {\it Higher order expansions for the probabilistic local Cauchy theory of the cubic nonlinear Schr\"odinger equation on $\R^3$}, 
Trans. Amer. Math. Soc. Ser. B 6 (2019), 114--160.  


\bibitem{Bony} J.-M.~Bony,  {\it Calcul symbolique et propagation des singularit\'es pour les \'equations aux d\'eriv\'ees partielles non lin\'eaires,} Ann. Sci. \'Ecole Norm. Sup.  14 (1981), 209--246.


\bibitem{BO96} J.~Bourgain,  {\it Invariant measures for the 2D-defocusing nonlinear Schr\"odinger equation},  Comm. Math. Phys. 176 (1996), 421--445. 



\bibitem{BTTz}
N.~Burq, L.~Thomann, N.~Tzvetkov, 
{\it Global infinite energy solutions for the cubic wave equation,} Bull. Soc. Math. France 143 (2015), no. 2, 301--313.


\bibitem{BTT1}
N.~Burq, L.~Thomann, N.~Tzvetkov, 
{\it Remarks on the Gibbs measures for nonlinear dispersive equations}, 
 Ann. Fac. Sci. Toulouse Math.
 27  (2018), no. 3, 527--597.


\bibitem{BT1} N. Burq,  N. Tzvetkov, {\it Random data Cauchy theory for supercritical wave equations I. Local theory}, Invent. Math. 173 (2008), 449--475.
 
\bibitem{BT-JEMS}  N. Burq, N. Tzvetkov, {\it  Probabilistic well-posedness for the cubic wave equation}, J. Eur. Math. Soc. 16 (2014), 1--30.


\bibitem{CC}
R.~Catellier, K.~Chouk, 
{\it Paracontrolled distributions and the 3-dimensional stochastic quantization equation},
  Ann. Probab. 46 (2018), no. 5, 2621--2679.



\bibitem{Ch1}
A.~Chapouto, 
{\it A remark on the well-posedness of the modified KdV equation in the Fourier-Lebesgue spaces}, 
arXiv:1911.00551 [math.AP].

\bibitem{Ch2}
A.~Chapouto, 
{\it A refined well-posedness result for the modified KdV equation in the Fourier-Lebesgue spaces},
arXiv:2006.15671 [math.AP].


\bibitem{CCT}
M.~Christ, J.~Colliander, T.~Tao, 
{\it Ill-posedness for nonlinear Schr\"odinger and wave equations}, 
arXiv:math/0311048 [math.AP].



\bibitem{CKSTT} J. Colliander, M. Keel, G. Staffilani, H. Takaoka, T. Tao,   {\it Almost conservation laws and global rough solutions to a nonlinear Schr\"odinger equation}, Math. Res. Lett. 9 (2002), 659--682.



\bibitem{DPD2}
G.~Da Prato, A.~Debussche, 
{\it Strong solutions to the stochastic quantization equations,} Ann. Probab. 31 (2003), no. 4, 1900--1916.




\bibitem{DNY2}
Y.~Deng, A.~Nahmod, H.~Yue,
{\it Random tensors, propagation of randomness, and nonlinear dispersive equations}, 
arXiv:2006.09285 [math.AP].


\bibitem{EJS}
W.~E, 
A.~Jentzen, H.~Shen, 
{\it Renormalized powers of Ornstein-Uhlenbeck processes and well-posedness of stochastic Ginzburg-Landau equations,}
Nonlinear Anal. 142 (2016), 152--193. 


\bibitem{FOk}
J.~Forlano, M.~Okamoto, 
{\it A remark on norm inflation for nonlinear wave equations}, 
 Dyn. Partial Differ. Equ.
 17 (2020), no. 4, 361--381.

\bibitem{FH}
P.~Friz, M.~Hairer, 
{\it A course on rough paths. With an introduction to regularity structures,} Universitext. Springer, Cham, 2014. xiv+251 pp. 


\bibitem{GV}
J.~Ginibre, G.~Velo, 
{\it Generalized Strichartz inequalities for the wave equation,} J. Funct. Anal. 133 (1995),
50--68.


\bibitem{GIP} M.~Gubinelli, P.~Imkeller, P.~Perkowski, {\it Paracontrolled distributions and singular PDEs},  Forum  Math.  Pi 3 (2015), e6, 75 pp.



\bibitem{GKO} M.~Gubinelli, H.~Koch, T.~Oh
 {\it Renormalization of the two-dimensional stochastic nonlinear wave equation},   Trans. Amer. Math. Soc. 370 (2018), 7335--7359.


\bibitem{GKO2} M.~Gubinelli, H.~Koch, T.~Oh
{\it Paracontrolled approach to the three-dimensional stochastic nonlinear wave equation with quadratic nonlinearity},
arXiv:1811.07808 [math.AP].



\bibitem{GKOT}
M.~Gubinelli, H.~Koch, T.~Oh, L.~Tolomeo,
{\it Global dynamics for  the two-dimensional stochastic nonlinear wave equations,}
arXiv:2005.10570 [math.AP].



\bibitem{GO}
Z.~Guo, T.~Oh,
{\it  Non-existence of solutions for the periodic cubic nonlinear Schr\"dinger equation below $L^2$}, 
Internat. Math. Res. Not. 2018, no.6, 1656--1729. 

\bibitem{H} M.~Hairer, {\it  A theory of regularity structures},  Invent. Math. 198  (2014), 269--504.


\bibitem{HairerM}
M.~Hairer, 
K.~Matetski, 
{\it Discretisations of rough stochastic PDEs,} Ann. Probab. 46 (2018), no. 3, 1651--1709.

\bibitem{HRW}
M.~Hairer, M.~Ryser, H.~Weber, 
{\it Triviality of the 2D stochastic Allen-Cahn equation,} Electron. J. Probab. 17 (2012), no. 39, 14 pp.


\bibitem{HairerS}
M.~Hairer, H.~Shen, 
{\it The dynamical sine-Gordon model,} Comm. Math. Phys. 341 (2016), no. 3, 933--989.



\bibitem{KeelTao}
M.~Keel, T.~Tao, {\it Endpoint Strichartz estimates},  Amer. J. Math. 120 (1998), no. 5, 955--980.





\bibitem{LS} H.~Lindblad, C.~Sogge,  {\it On existence and scattering with minimal regularity for semilinear wave equations}, J. Funct. Anal. 130 (1995), 357--426.


\bibitem{Lyons}
T.~Lyons, 
{\it Differential equations driven by rough signals,}
 Rev. Mat. Iberoamericana 14 (1998), no. 2, 215--310.

\bibitem{McKean}
H.P.~McKean, 
{\it Statistical mechanics of nonlinear wave equations. IV. Cubic Schr\"odinger,} 
 Comm. Math. Phys. 168 (1995), no. 3, 479--491. 
 {\it Erratum: Statistical mechanics of nonlinear wave equations. IV. Cubic Schr\"odinger}, Comm. Math. Phys. 173 (1995), no. 3, 675.







\bibitem{MW2}
J.-C.~Mourrat, H.~Weber,
{\it Global well-posedness of the dynamic $\Phi^4$ model in the plane,}
 Ann. Probab. 45 (2017), no. 4, 2398--2476.



\bibitem{MW} J.C. Mourrat, H. Weber, {\it The dynamic $\Phi^4_3$ model comes down from infinity}, Comm. Math. Phys. 356 (2017) 673--753.



\bibitem{MWX}
J.-C.~Mourrat, H.~Weber, W.~Xu,
{\it Construction of $\Phi^4_3$ diagrams for pedestrians,}
 From particle systems to partial differential equations, 1--46, Springer Proc. Math. Stat., 209, Springer, Cham, 2017.


\bibitem{Nelson2}
E.~Nelson, 
{\it A quartic interaction in two dimensions}, 
 1966 Mathematical Theory of Elementary Particles (Proc. Conf., Dedham, Mass., 1965) pp. 69--73 M.I.T. Press, Cambridge, Mass.


\bibitem{OOcomp}
T.~Oh, M.~Okamoto,
{\it  Comparing the stochastic nonlinear wave and heat equations: a case study},
arXiv:1908.03490 [math.AP].

\bibitem{OOR}
T.~Oh, M.~Okamoto, T.~Robert,
{\it  A remark on triviality for the two-dimensional stochastic nonlinear wave equation}, Stochastic Process. Appl. 130 (2020), no. 9, 5838--5864. 


\bibitem{OOTz}
T.~Oh, M.~Okamoto, N.~Tzvetkov,
{\it  Uniqueness and non-uniqueness of the Gaussian free field evolution under the two-dimensional Wick ordered cubic wave equation,} preprint.




\bibitem{ORSW}
T.~Oh, T.~Robert, P.~Sosoe, Y.~Wang,
{\it  On the two-dimensional hyperbolic stochastic sine-Gordon equation}, Stoch. Partial Differ. Equ. Anal. Comput. (2020), 32 pages. https://doi.org/10.1007/s40072-020-00165-8 





\bibitem{OTh2} T.~Oh, L.~Thomann, 
{\it  Invariant Gibbs measures for the 2-$d$ defocusing nonlinear wave equations}, 
 Ann. Fac. Sci. Toulouse Math.
 29 (2020), no. 1, 1--26. 

\bibitem{OTW} T.~Oh,  N. Tzvetkov, Y. Wang,   {\it Solving the 4NLS with white noise initial data}, 
Forum Math. Sigma.
 8 (2020), e48, 63 pp. 



\bibitem{OW}
T.~Oh, Y.~Wang,
{\it  Global well-posedness of the periodic cubic fourth order NLS in negative Sobolev spaces},
 Forum Math. Sigma 6 (2018), e5, 80 pp. 


\bibitem{Simon} B.~Simon,  {\it  The $P(\varphi)_2$ Euclidean (quantum) field theory,} Princeton Series in Physics. Princeton University Press, Princeton, N.J., 1974. xx+392 pp.

\bibitem{TTz} L.~Thomann, N.~Tzvetkov,  {\it Gibbs measure for the periodic derivative nonlinear Schr\"odinger equation}, Nonlinearity 23 (2010), no. 11, 2771--2791.


\bibitem{TzBO}
N.~Tzvetkov, 
{\it Construction of a Gibbs measure associated to the periodic Benjamin-Ono equation,}
 Probab. Theory Related Fields 146 (2010), no. 3-4, 481--514. 

\bibitem{tz-cime} 
N.~Tzvetkov, {\it Random data wave equations},  
 Singular random dynamics, 221--313, Lecture Notes in Math., 2253, Fond. CIME/CIME Found. Subser., Springer, Cham, 2019.




\bibitem{Xia}
B.~Xia, 
{\it Generic ill-posedness for wave equation of power type on 3D torus}, 
 	arXiv:1507.07179 [math.AP].


\end{thebibliography}
\end{document}